\documentclass[reqno]{amsart}

\usepackage[paper=letterpaper,margin=.9in]{geometry}

\usepackage[title]{appendix}
\usepackage[utf8]{inputenc}
\usepackage{amsmath,amsthm,amssymb,amsfonts}
\usepackage[pdfborder={0 0 0}]{hyperref}
\hypersetup{colorlinks,linkcolor={blue},citecolor={blue},urlcolor={red}} 
\usepackage{xcolor}
\usepackage{tikz}
\usepackage{adjustbox}
\usetikzlibrary{arrows}
\usepackage{mathtools}
\usepackage{theoremref}
\usepackage{bm}
\usepackage{graphicx}
\usepackage{subfig}
\usepackage{float}
\usepackage{bbm}
\usepackage{stmaryrd}
\usepackage{array}
\usepackage{enumitem}
\usepackage{stmaryrd}

\newtheorem{theorem}{Theorem}[section]
\newtheorem{lemma}[theorem]{Lemma}
\newtheorem{prop}[theorem]{Proposition}
\newtheorem{corollary}[theorem]{Corollary}
\newtheorem{definition}[theorem]{Definition}
\newtheorem{assumption}[theorem]{Assumption}
\newtheorem{rmk}[theorem]{Remark}
\theoremstyle{definition}

\numberwithin{equation}{section}

\newcommand{\be}{\begin{equation}}
\newcommand{\ee}{\end{equation}}
\newcommand{\de}{\begin{align*}}
\newcommand{\fe}{\end{align*}}

\newcommand{\Z}{\mathbb{Z}}
\newcommand{\mmin}{\mathfrak{min}}
\newcommand{\prods}{\sideset{}{^*} \prod}
\newcommand{\blue}[1]{\textcolor{blue}{#1}}

\newcommand{\LL}{\mathsf{L}}

\newcommand{\e}{\epsilon}
\newcommand{\mfs}{\mathfrak{s}}
\newcommand{\mfi}{\mathfrak{i}}
\newcommand{\mfj}{\mathfrak{j}}
\newcommand{\mfk}{\mathfrak{k}}
\newcommand{\mfl}{\mathfrak{l}}
\newcommand{\bl}{\mathbf{l}}
\newcommand{\bk}{\mathbf{k}}
\newcommand{\bi}{\mathbf{i}}
\newcommand{\bj}{\mathbf{j}}

\definecolor{mintbg}{rgb}{.63,.79,.95}

\colorlet{lightmintbg}{mintbg!40}

\begin{document}
	
	\title{Strong law of large numbers for the stochastic six vertex model}
	\author{Hindy Drillick and Yier Lin}
	\address[Hindy Drillick]{\ \hspace{50pt}Department of Mathematics, Columbia University}
	\email{hindy.drillick@columbia.edu}
	\address[Yier Lin]{\ \hspace{72pt}Department of Statistics, University of Chicago}
	\email{ylin10@uchicago.edu}
	%last update time
	%\date{Last modified on \today}
	\begin{abstract}
 We consider the inhomogeneous stochastic six vertex model with periodicity starting from step initial data. We prove that it converges almost surely to a deterministic limit shape. For the proof, we map the stochastic six vertex model to a deformed version of the discrete Hammersley process \cite{seppalainen1997increasing, basdevant2015discrete}. Then we construct a colored version of the model and apply Liggett’s superadditive ergodic theorem. The construction of the colored model includes a new idea using a Boolean-type product, which generalizes and simplifies the method used in \cite{drillick2022hydrodynamics}.
	\end{abstract}
	\maketitle

	\section{Introduction}
	\subsection{The S6V model}
	\label{sec:s6v}
	The stochastic six vertex (S6V) model \cite{gwa1992six} is a classical model in two-dimensional statistical physics. It is a stochastic version of the six vertex model \cite{lieb1967residual, baxter2016exactly}, which is a natural model for crystal lattices with hydrogen bonds. We associate six possible configurations to each lattice point on the first quadrant, as in Figure \ref{fig:s6v}. We view the lines from the south and the west as \emph{inputs} and the lines to the north and the east as \emph{outputs}. The model is \emph{stochastic} in the sense that when we fix the inputs, the weights of possible
	configurations with those inputs sum up to 1. The configurations chosen for neighboring vertices need to be compatible so that the output lines of a vertex are the input lines of the vertices to its immediate north and east.

	By \cite{borodin2016stochastic}, we know that the S6V model belongs to the Kardar-Parisi-Zhang (KPZ) universality class---a class of models that exhibit universal statistical behavior in their large time/large-scale limits, see \cite{kardar1986dynamic, quastel2011introduction, corwin2012kardar} for an overview of this topic. Some recent works on the S6V model include \cite{aggarwal2017convergence, aggarwal2018current, borodin2019stochastic, lin2019markov, shen2019stochastic, aggarwal2020limit, corwin2020stochastic, dimitrov2020two, kuan2021short, lin2022classification} and references therein.

	In this paper, we view the S6V model as a stochastic path ensemble on the first quadrant with the boundary condition that there is a line entering each of the vertices in $\{1\} \times \mathbb{Z}_{\geq 1}$ from the left, and no lines entering $\mathbb{Z}_{\geq 1} \times \{1\}$ from the bottom. This boundary data is called \textit{step initial data.} Fix a collection of parameters $\{b_1(i, j), b_2 (i, j)\}_{i, j}$ that take values in $[0, 1]$. Starting from $(1, 1)$, we tile each vertex with one of the six vertex configurations in Figure \ref{fig:s6v}, where we only consider configurations whose input lines match the input lines of the given vertex. We assign an allowed configuration with probability given by the weight of the configuration. This tiling construction then progresses sequentially in the linear order $(2, 1), (1, 2), (3, 1), (2, 2), (1, 3), \dots$ to the entire quadrant. Note that we can view the stochastic path ensemble as a family of upright paths ordered in the northwest direction, as depicted in the left panel of Figure \ref{fig:sampling}. 
	
	Let $h$ be the height function of the S6V model defined on $\mathbb{Z}_{\geq 0}^2$ where $h(x, y)$ records the number of lines that pass through or to the right of $(x, y)$, see the left panel of Figure \ref{fig:sampling}. 
	\begin{figure}[ht]
		\centering
		\begin{tabular}{|c|c|c|c|c|c|c|}
			\hline
			Type & I & II & III & IV & V & VI \\
			\hline
			\begin{tikzpicture}[scale = 1.5]
			%----
			\draw[fill][white] (0.5, 0) circle (0.05);
			\draw[thick][white] (0, 0) -- (1,0);
			\draw[thick][white] (0.5, -0.5) -- 
			(0.5,0.5);
			\node at (0.5, 0) {Configuration};
			\end{tikzpicture}
			&
			\begin{tikzpicture}[scale = 1.2]
			%----
			\draw[fill] (0.5, 0) circle (0.05);
			\draw[thick] (0, 0) -- (1,0);
			\draw[thick] (0.5, -0.5) -- (0.5,0.5);
			\end{tikzpicture}
			&
			\begin{tikzpicture}[scale = 1.2]
			%----
			\draw[thick][white] (0, 0) -- (1,0);
			\draw[thick][white] (0.5, -0.5) -- (0.5,0.5);
			\draw[fill] (0.5, 0) circle (0.05);
			\end{tikzpicture}
			&
			\begin{tikzpicture}[scale = 1.2]
			%----
			\draw[thick][white] (0, 0) -- (1,0);
			\draw[thick] (0.5, -0.5) -- (0.5,0.5);
			\draw[fill] (0.5, 0) circle (0.05);
			\end{tikzpicture}
			&
			\begin{tikzpicture}[scale = 1.2]
			%----
			\draw[thick][white] (0, 0) -- (0.5,0);
			\draw[thick][white] (0.5, 0) -- (0.5, 0.5);
			\draw[thick] (0.5, 0) -- (1, 0);
			\draw[thick] (0.5, -0.5) -- (0.5, 0);
			(0.5,0.5);
			\draw[fill] (0.5, 0) circle (0.05);
			\end{tikzpicture}
			&
			\begin{tikzpicture}[scale = 1.2]
			%----
			\draw[thick] (0, 0) -- (1,0);
			\draw[thick][white] (0.5, -0.5) -- (0.5,0.5);
			\draw[fill] (0.5, 0) circle (0.05);
			\end{tikzpicture}
			&
			\begin{tikzpicture}[scale = 1.2]
			%----
			\draw[thick] (0, 0) -- (0.5,0);
			\draw[thick] (0.5, 0) -- (0.5, 0.5);
			\draw[thick][white] (0.5, 0) -- (1, 0);
			\draw[thick][white] (0.5, -0.5) -- (0.5, 0);
			(0.5,0.5);
			\draw[fill] (0.5, 0) circle (0.05);
			\end{tikzpicture}
			\\
			\hline
			Weight 
			%(first parametrization)
			& 1 & 1 & $b_1 (i, j)$ & $1- b_1 (i, j)$ & $b_2 (i, j)$ & $1-b_2 (i, j)$\\
			\hline
		\end{tabular}
		\caption{Six types of configurations for the S6V model}
		\label{fig:s6v}
	\end{figure}
	\begin{assumption} \label{assumption}
		There exist periodicities $I, J \in \mathbb{Z}_{\geq 1}$ such that $b_k (x+I, y) = b_k (x, y)$ and $b_k (x, y+J) = b_k (x, y)$ for arbitrary $x, y$ and $k \in \{1, 2\}$.
	\end{assumption}
	\begin{theorem}\label{thm:main}
		There exists a Lipschitz function $g$ such that with probability $1$,
		\begin{equation}\label{eq:asconverge}
		\lim_{n \to \infty} \frac{h(\lfloor nx \rfloor, \lfloor ny\rfloor)}{n} = g(x, y), \qquad \forall x, y \in \mathbb{R}_{\geq 0}.
		\end{equation}
	\end{theorem}
	\begin{rmk}
		The S6V model considered here is inhomogeneous and not integrable in most cases. 
		The parameters $b_1$ and $b_2$ have periodicity $I$ in space and periodicity $J$ in time. Hence, there are $IJ$ degrees of freedom. If $I, J = 1$, the S6V model is homogeneous and its limit shape has been explicitly derived in \cite[Theorem 1.1]{borodin2016stochastic} and \cite[Theorem 5.1]{aggarwal2020limit} (also see \cite{reshetikhin2018limit} for a derivation of the limit shape using a variational principle). Note that the derivations of explicit limit shapes in these two papers rely respectively on integrability and the existence of explicit stationary distributions for the model. For general $I, J$, due to the lack of these properties, we don't think that there is an explicit formula for $g$.   
	\end{rmk}
	\begin{rmk}\label{rmk:limitshape}
		Take $I = J = 1$ and consider the homogeneous S6V model where $b_1 (i, j)$ and $b_2 (i, j)$ do not depend on $i, j$. By \cite[Theorem 1.1]{borodin2016stochastic} or \cite[Theorem 5.1]{aggarwal2020limit}, if $b_1 \leq b_2$, we have
\begin{equation}\label{eq:hlimit}
		g(x, y) = 
		\begin{cases}
		\frac{\big(\sqrt{y(1-b_1)} - \sqrt{x(1-b_2)}\big)^2}{b_2 - b_1} &\qquad \frac{1 - b_2}{1-b_1} < \frac{x}{y} < \frac{1-b_1}{1-b_2}, \\
		0 &\qquad \frac{x}{y} \geq \frac{1-b_1}{1-b_2},\\
		y - x &\qquad \frac{x}{y} \leq \frac{1-b_2}{1-b_1}.
		\end{cases}
	\end{equation}
If $b_1 \geq b_2$, we have 
\begin{equation*}
g(x, y) = 
\begin{cases}
0 &\qquad x \geq y,\\
y-x &\qquad x \leq y.
\end{cases}
\end{equation*}
With Theorem \ref{thm:main}, we strengthen the convergence in probability in \cite{borodin2016stochastic, aggarwal2020limit} to the almost sure level.
	\end{rmk}
	\subsection{The complemented S6V model}
	To prove Theorem \ref{thm:main}, it is convenient to study the S6V model after \emph{horizontal complementation}. If there is a horizontal line, we erase it; if there is no horizontal line, we add it, see Figure \ref{fig:sampling} for an example. This procedure converts an ensemble of upright paths to an ensemble of downright paths. Note that this model has appeared in \cite{aggarwal2021deformed} without a specific name, and we call it the \emph{complemented S6V (CS6V) model}.

	Let $H$ be the height function of the CS6V model defined on $\mathbb{Z}_{\geq 0}^2$, where $H(x, y)$ records the number of lines that pass through or to the left of $(x, y)$. 
	It is straightforward that $H(x, y) = y - h(x, y)$, 
	see the right panel of Figure \ref{fig:sampling}. Note that by interpreting the stochastic paths as a collection of downright paths ordered in the northeast direction, these downright paths are the level lines of the height function. 
	\begin{figure}[ht]
		\centering
		\begin{tikzpicture}

		\begin{scope}[xshift = 0cm, scale = 0.7]
		\foreach \x in {0, 1, 2, 3, 4, 5, 6, 7}
		{\draw[lightmintbg] (\x, 0) -- (\x, 7);
			\draw[lightmintbg] (0, \x) -- (7, \x);
		}
		\draw[lightmintbg, ->] (7, 0) -- (7.5, 0);
		\draw[lightmintbg, ->] (0, 7) -- (0, 7.5);
		\foreach \x in {1, 2, 3, 4, 5, 6, 7}
		{
			\foreach \y in {1, 2, 3, 4, 5, 6, 7}
			\draw[fill] (\x, \y) circle (0.05);
		}
		\draw (0, 1) -- (4, 1);
		\draw (0, 2) -- (1, 2);
		\draw (0, 3) -- (2, 3);
		\draw (4, 3) -- (7.5, 3);
		\draw (0, 4) -- (4, 4);
		\draw (0, 5) -- (3, 5) -- (3, 7.5);
		\draw (0, 6) -- (7.5, 6);
		\draw (0, 7) -- (5, 7) -- (5, 7.5);
		\draw (1, 2) -- (1, 7.5);
		\draw (2, 3) -- (2, 7.5);
		\draw (3, 5) -- (3, 5.5);
		\draw (4, 1) -- (4, 3);
		\draw (4, 4) -- (4, 7.5);
		%\node at (0.6, 0.6) {$(0, 0)$};
		\foreach \x in {0,1,2,3,4,5,6,7} 
		\node at (\x+ 0.2, 0.2) {\footnotesize $\blue{0}$}; 
		\foreach \x in {0,1,2,3} 
		\node at (\x+ 0.2, 1.2) {\footnotesize $\blue{1}$}; 
		\foreach \x in {4,5,6,7} 
		\node at (\x+ 0.2, 1.2) {\footnotesize $\blue{0}$}; 
		\foreach \x in {0} 
		\node at (\x+ 0.2, 2.2) {\footnotesize $\blue{2}$};
		\foreach \x in {1,2,3} 
		\node at (\x+ 0.2, 2.2) {\footnotesize $\blue{1}$};
		\foreach \x in {4,5,6,7} 
		\node at (\x+ 0.2, 2.2) {\footnotesize $\blue{0}$};
		\foreach \x in {0} 
		\node at (\x+ 0.2, 3.2) {\footnotesize $\blue{3}$};
		\foreach \x in {1} 
		\node at (\x+ 0.2, 3.2) {\footnotesize $\blue{2}$};
		\foreach \x in {2,3,4,5,6,7} 
		\node at (\x+ 0.2, 3.2) {\footnotesize $\blue{1}$};
		\foreach \x in {0} 
		\node at (\x+ 0.2, 4.2) {\footnotesize $\blue{4}$};
		\foreach \x in {1} 
		\node at (\x+ 0.2, 4.2) {\footnotesize $\blue{3}$};
		\foreach \x in {2, 3} 
		\node at (\x+ 0.2, 4.2) {\footnotesize $\blue{2}$};
		\foreach \x in {4, 5, 6, 7} 
		\node at (\x+ 0.2, 4.2) {\footnotesize $\blue{1}$};
		\foreach \x in {0} 
		\node at (\x+ 0.2, 5.2) {\footnotesize $\blue{5}$};
		\foreach \x in {1} 
		\node at (\x+ 0.2, 5.2) {\footnotesize $\blue{4}$};
		\foreach \x in {2} 
		\node at (\x+ 0.2, 5.2) {\footnotesize $\blue{3}$};
		\foreach \x in {3} 
		\node at (\x+ 0.2, 5.2) {\footnotesize $\blue{2}$};
		\foreach \x in {4,5,6,7} 
		\node at (\x+ 0.2, 5.2) {\footnotesize $\blue{1}$};
		\foreach \x in {0} 
		\node at (\x+ 0.2, 6.2) {\footnotesize $\blue{6}$};
		\foreach \x in {1} 
		\node at (\x+ 0.2, 6.2) {\footnotesize $\blue{5}$};
		\foreach \x in {2} 
		\node at (\x+ 0.2, 6.2) {\footnotesize $\blue{4}$};
		\foreach \x in {3} 
		\node at (\x+ 0.2, 6.2) {\footnotesize $\blue{3}$};
		\foreach \x in {4,5,6,7} 
		\node at (\x+ 0.2, 6.2) {\footnotesize $\blue{2}$};
		\foreach \x in {0} 
		\node at (\x+ 0.2, 7.2) {\footnotesize $\blue{7}$};
		\foreach \x in {1} 
		\node at (\x+ 0.2, 7.2) {\footnotesize $\blue{6}$};	
		\foreach \x in {2} 
		\node at (\x+ 0.2, 7.2) {\footnotesize $\blue{5}$};	
		\foreach \x in {3} 
		\node at (\x+ 0.2, 7.2) {\footnotesize $\blue{4}$};
		\foreach \x in {4} 
		\node at (\x+ 0.2, 7.2) {\footnotesize $\blue{3}$};
		\foreach \x in {5,6,7} 
		\node at (\x+ 0.2, 7.2) {\footnotesize $\blue{2}$};
		\node at (-0.3, -0.3) {\footnotesize $(0, 0)$};
		\end{scope}

		\begin{scope}[xshift = 7cm, scale = 0.7]
		
		\foreach \x in {0, 1, 2, 3, 4, 5, 6, 7}
		{\draw[lightmintbg] (\x, 0) -- (\x, 7);
			\draw[lightmintbg] (0, \x) -- (7, \x);
		}
		\foreach \x in {1, 2, 3, 4, 5, 6, 7}
		{
			\foreach \y in {1, 2, 3, 4, 5, 6, 7}
			\draw[fill] (\x, \y) circle (0.05);
		}
		\draw[lightmintbg, ->] (7, 0) -- (7.5, 0);
		\draw[lightmintbg, ->] (0, 7) -- (0, 7.5);
		%\draw[->] (0.5, 0.5) -- (5.5, 0.5);
		%\draw[->] (0.5, 0.5) -- (0.5, 5.5);
		%\draw (0.5, 1) -- (1, 1) -- (1, 3) -- (2, 3) -- (2, 4) -- (4, 4) -- (4, 5.5);
		%\draw (2, 0.5) -- (2, 2) -- (5.5, 2);
		%\draw (4, 0.5) -- (4, 3) -- (5.5, 3);
		%\draw (0.5, 5) -- (3, 5) -- (3, 5.5);
		%\draw (1, 1) -- (5.5, 1);
		%	\draw (0.5, 3) -- (1, 3);
		%\draw (2, 0.5) -- (2, 2);
		%\draw (0.5, 4) -- (2, 4);
		\draw (1, 2) -- (7.5, 2);
		\draw (2, 3) -- (4, 3);
		\draw (4, 4) -- (7.5, 4); 
		\draw (3, 5) -- (7.5, 5);
		\draw (1, 2) -- (1, 7.5);
		\draw (2, 3) -- (2, 7.5);
		\draw (7.5, 1) -- (4, 1) -- (4, 3);
		\draw (3, 5) -- (3, 7.5);
		\draw (4, 4) -- (4, 7.5);
		\draw (5, 7.5) -- (5, 7) -- (7.5, 7);
		%	\node at (0.6, 0.6) {$(0, 0)$};
		\foreach \x in {0,1,2,3,4,5,6,7} 
		\node at (\x+ 0.2, 0.2) {\footnotesize $\blue{0}$}; 
		\foreach \x in {0,1,2,3} 
		\node at (\x+ 0.2, 1.2) {\footnotesize $\blue{0}$}; 
		\foreach \x in {4,5,6,7} 
		\node at (\x+ 0.2, 1.2) {\footnotesize $\blue{1}$}; 
		\foreach \x in {0} 
		\node at (\x+ 0.2, 2.2) {\footnotesize $\blue{0}$};
		\foreach \x in {1,2,3} 
		\node at (\x+ 0.2, 2.2) {\footnotesize $\blue{1}$};
		\foreach \x in {4,5,6,7} 
		\node at (\x+ 0.2, 2.2) {\footnotesize $\blue{2}$};
		\foreach \x in {0} 
		\node at (\x+ 0.2, 3.2) {\footnotesize $\blue{0}$};
		\foreach \x in {1} 
		\node at (\x+ 0.2, 3.2) {\footnotesize $\blue{1}$};
		\foreach \x in {2,3,4,5,6,7} 
		\node at (\x+ 0.2, 3.2) {\footnotesize $\blue{2}$};
		\foreach \x in {0} 
		\node at (\x+ 0.2, 4.2) {\footnotesize $\blue{0}$};
		\foreach \x in {1} 
		\node at (\x+ 0.2, 4.2) {\footnotesize $\blue{1}$};
		\foreach \x in {2, 3} 
		\node at (\x+ 0.2, 4.2) {\footnotesize $\blue{2}$};
		\foreach \x in {4, 5, 6, 7} 
		\node at (\x+ 0.2, 4.2) {\footnotesize $\blue{3}$};
		\foreach \x in {0} 
		\node at (\x+ 0.2, 5.2) {\footnotesize $\blue{0}$};
		\foreach \x in {1} 
		\node at (\x+ 0.2, 5.2) {\footnotesize $\blue{1}$};
		\foreach \x in {2} 
		\node at (\x+ 0.2, 5.2) {\footnotesize $\blue{2}$};
		\foreach \x in {3} 
		\node at (\x+ 0.2, 5.2) {\footnotesize $\blue{3}$};
		\foreach \x in {4,5,6,7} 
		\node at (\x+ 0.2, 5.2) {\footnotesize $\blue{4}$};
		\foreach \x in {0} 
		\node at (\x+ 0.2, 6.2) {\footnotesize $\blue{0}$};
		\foreach \x in {1} 
		\node at (\x+ 0.2, 6.2) {\footnotesize $\blue{1}$};
		\foreach \x in {2} 
		\node at (\x+ 0.2, 6.2) {\footnotesize $\blue{2}$};
		\foreach \x in {3} 
		\node at (\x+ 0.2, 6.2) {\footnotesize $\blue{3}$};
		\foreach \x in {4,5,6,7} 
		\node at (\x+ 0.2, 6.2) {\footnotesize $\blue{4}$};
		\foreach \x in {0} 
		\node at (\x+ 0.2, 7.2) {\footnotesize $\blue{0}$};
		\foreach \x in {1} 
		\node at (\x+ 0.2, 7.2) {\footnotesize $\blue{1}$};	
		\foreach \x in {2} 
		\node at (\x+ 0.2, 7.2) {\footnotesize $\blue{2}$};	
		\foreach \x in {3} 
		\node at (\x+ 0.2, 7.2) {\footnotesize $\blue{3}$};
		\foreach \x in {4} 
		\node at (\x+ 0.2, 7.2) {\footnotesize $\blue{4}$};
		\foreach \x in {5,6,7} 
		\node at (\x+ 0.2, 7.2) {\footnotesize $\blue{5}$};
		\node at (-0.3, -0.3) {\footnotesize $(0, 0)$};
		\end{scope}
		
		\end{tikzpicture}
		\caption{\textbf{Left panel:} A sampling of the S6V model on the first quadrant, with the blue numbers denoting the height function. \textbf{Right panel:} The CS6V model obtained by horizontally complementing the S6V model.}
		\label{fig:sampling}
	\end{figure}
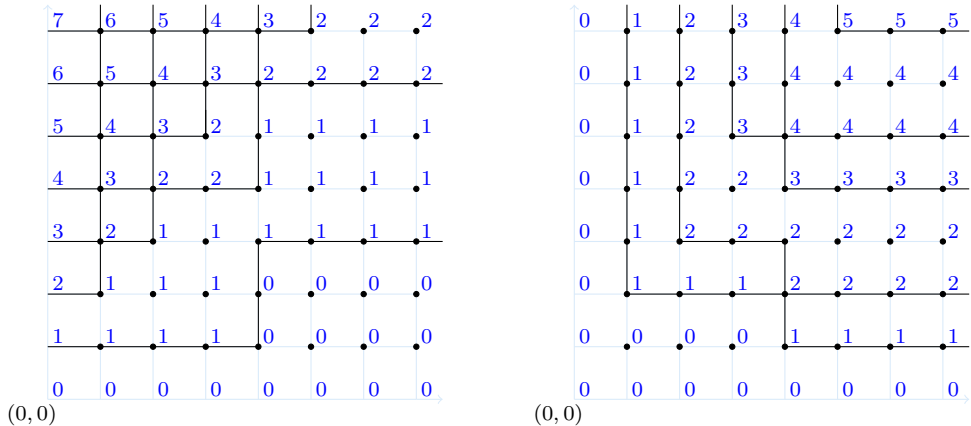

	There is another way to sample the CS6V model on the first quadrant that is more similar to how we defined the S6V model above. Consider the six vertex configurations in Figure \ref{fig:cs6v}. These configurations are obtained by horizontally complementing the configurations in Figure \ref{fig:s6v}. We consider \emph{empty boundary data}---the boundary data obtained by horizontally complementing the step initial data. In other words, no lines enter the quadrant from either the left boundary $\{1\} \times \mathbb{Z}_{\geq 1}$ or the bottom boundary $\mathbb{Z}_{\geq 1} \times \{1\}$. We sequentially tile the vertices $(1, 1), (2, 1), (1, 2), (3, 1), (2, 2), (1, 3), \dots$ with one of the six vertex configurations. We can view the tiling as a family of downright paths ordered in the northeast direction. This path ensemble has the same law as the one obtained by complementing the S6V model.
	\begin{figure}[ht]
		\centering
		%	\begin{adjustbox}
		\begin{tabular}{|c|c|c|c|c|c|c|}
			\hline
			Type & I & II & III & IV & V & VI \\
			\hline
			\begin{tikzpicture}[scale = 1.5]
			%----
			\draw[fill][white] (0.5, 0) circle (0.05);
			\draw[thick][white] (0, 0) -- (1,0);
			\draw[thick][white] (0.5, -0.5) -- 
			(0.5,0.5);
			\node at (0.5, 0) {Configuration at $(i, j)$};
			\end{tikzpicture}
			&
			\begin{tikzpicture}[scale = 1.2]
			%----
			\draw[thick][white] (0, 0) -- (1,0);
			\draw[thick] (0.5, -0.5) -- (0.5,0.5);
			\draw[fill] (0.5, 0) circle (0.05);
			\end{tikzpicture}
			&
			\begin{tikzpicture}[scale = 1.2]
			%----
			\draw[thick] (0, 0) -- (1,0);
			\draw[thick][white] (0.5, -0.5) -- (0.5,0.5);
			\draw[fill] (0.5, 0) circle (0.05);
			\end{tikzpicture}
			&
			\begin{tikzpicture}[scale = 1.2]
			%----
			\draw[thick] (0, 0) -- (1,0);
			\draw[thick] (0.5, -0.5) -- (0.5,0.5);
			\draw[fill] (0.5, 0) circle (0.05);
			\end{tikzpicture}
			&
			\begin{tikzpicture}[scale = 1.2]
			%----
			\draw[thick] (0, 0) -- (0.5,0);
			\draw[thick][white] (0.5, 0) -- (0.5, 0.5);
			\draw[thick][white] (0.5, 0) -- (1, 0);
			\draw[thick] (0.5, -0.5) -- (0.5, 0);
			(0.5,0.5);
			\draw[fill] (0.5, 0) circle (0.05);
			\end{tikzpicture}
			&
			\begin{tikzpicture}[scale = 1.2]
			%----
			\draw[thick][white] (0, 0) -- (1,0);
			\draw[thick][white] (0.5, -0.5) -- (0.5,0.5);
			\draw[fill] (0.5, 0) circle (0.05);
			\end{tikzpicture}
			&
			\begin{tikzpicture}[scale = 1.2]
			%----
			\draw[thick] (0.5, 0) -- (1,0);
			\draw[thick] (0.5, 0) -- (0.5, 0.5);
			\draw[thick][white] (0, 0) -- (0.5, 0);
			\draw[thick][white] (0.5, -0.5) -- (0.5, 0);
			(0.5,0.5);
			\draw[fill] (0.5, 0) circle (0.05);
			\end{tikzpicture}
			\\
			\hline
			Weight 
			%(first parametrization)
			& 1 & 1 & $b_1(i, j)$ & $1- b_1 (i, j)$ & $b_2 (i, j)$ & $1 - b_2 (i, j)$\\
			\hline
		\end{tabular}
		%	\end{adjustbox}
		\caption{Six types of configurations for the CS6V model}
		\label{fig:cs6v}
	\end{figure}

	Due to the relationship $H(x, y) = y - h(x,y)$, proving Theorem \ref{thm:main} is equivalent to proving the following theorem. 
	\begin{theorem}\label{thm:maincs6v}
		There exists a Lipschitz function $g$ such that with probability $1$, 
		\begin{equation*}
		\lim_{n \to \infty} \frac{H(\lfloor nx \rfloor, \lfloor ny\rfloor)}{n} = y - g(x, y), \qquad \forall x, y \in \mathbb{R}_{\geq 0}.
		\end{equation*}
	\end{theorem}

	\subsection{The CS6V model and discrete Hammersley process}
 \label{sec:discreteH}
	\cite{seppalainen1997increasing} studied Ulam's problem on the planar lattice. This model is also referred to as the discrete Hammersley process in \cite{basdevant2015discrete, ciech2019order}. Consider the integer lattice on $\mathbb{Z}_{\geq 0}^2$ and a  random set $\xi$ of integer points chosen independently with probability $p$. Define the  partial order $\prec$ such that $(x_1, y_1) \prec (x_2, y_2)$ if and only if $x_1 < x_2$ and $y_1 < y_2$.

	Let $H^{\mathsf{d}}$ denote the height function defined on $\mathbb{Z}_{\geq 0}^2$: 
	\begin{equation*}
	H^{\mathsf{d}} (x, y) = \max\Big\{L: \text{ there exist integer points } (x_1, y_1) \prec \dots \prec (x_L, y_L) \in \xi \cap [1, x] \times [1, y] \Big\}.
	\end{equation*}
	As noticed in \cite{basdevant2015discrete}, we can construct a collection of downright paths which play the role of the level lines of $H^\mathsf{d}$. We first use a downright path to connect the minimal points of $\xi$ under the order $\prec$.  We remove these points from $\xi$ and connect the new minima to obtain the second line and so on. The height function starts from $H^\mathsf{d}(0, 0) = 0$ and whenever we cross a downright path from southwest to northeast, the height function increases by $1$, see Figure \ref{fig:dham}.
	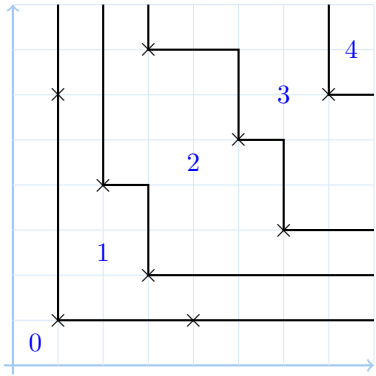
\begin{figure}[ht]
%		\begin{minipage}{0.48\textwidth}
			\begin{tikzpicture}[scale = 0.6]
			\draw[mintbg, thick, ->] (-0.2, 0) -- (8, 0);
			\draw[mintbg, thick, ->] (0, -0.2) -- (0, 8);
			\foreach \x in {1, 2, 3, 4, 5, 6, 7, 8}
			{
				\draw[lightmintbg] (\x, 0) -- (\x, 8);
				\draw[lightmintbg] (0, \x) -- (8, \x);
			}
			\node at (1, 1) {$\times$};
			\node at (1, 6) {$\times$};
			\node at (4, 1) {$\times$};
			\node at (2, 4) {$\times$};
			\node at (3, 2) {$\times$};
			\node at (3, 7) {$\times$};
			\node at (5, 5) {$\times$};
			\node at (6, 3) {$\times$};
			\node at (7, 6) {$\times$};
			\draw[thick] (1, 8) -- (1, 1) -- (8, 1);
			\draw[thick] (2, 8) -- (2, 4) -- (3, 4) -- (3, 2) -- (8, 2);
			\draw[thick] (3, 8) -- (3, 7) -- (5, 7) -- (5, 5) -- (6, 5) -- (6, 3) -- (8, 3);
			\draw[thick] (7, 8) -- (7, 6) -- (8, 6);
			\node at (0.5, 0.5) {$\blue{0}$};
			\node at (2, 2.5) {$\blue{1}$};
			\node at (4, 4.5) {$\blue{2}$};
			\node at (6, 6) {$\blue{3}$};
			\node at (7.5, 7) {$\blue{4}$};
			\end{tikzpicture}
			\caption{Illustration of the discrete Hammersley process. The $\times$'s are chosen independently on the lattice with probability $p$.}
			\label{fig:dham}
%		\end{minipage}
\end{figure}
	The following theorem shows that the CS6V model is a deformed version of the discrete Hammersley process. 
	\begin{theorem}\label{thm:cs6vdham}
		Take $b_1 = 0$ and $b_2 = 1-p$. The CS6V model degenerates to the discrete Hammersley process.  
	\end{theorem}
	\begin{proof}
		It suffices to show that the collection of stochastic lines in Figure \ref{fig:dham} has the same distribution as those in the CS6V model when $b_1 = 0$ and $b_2 = 1-p$. By the definition of the discrete Hammersley process, we can also sample the stochastic lines in Figure \ref{fig:dham} by sequentially sampling the configurations of the vertices in the linear order $(1, 1), (2, 1), (1, 2), (3, 1), (2, 2), (1, 3), \dots$. If there is only a single line entering a vertex from either the left or bottom, then the line continues. If there are lines entering from both the left and bottom, then they annihilate each other and form a corner. If no lines enter from the left and bottom, then with probability $1-b_2 = p$ the vertex is chosen to be a part of $\xi$, which corresponds to choosing a configuration of type VI from Figure \ref{fig:cs6v}. Otherwise, we choose a configuration of type V. This coincides with the sequential sampling rule for the CS6V model. 
	\end{proof}
	Using this, we recover the almost sure limit shape of the discrete Hammersley process established in \cite[Theorem 1]{seppalainen1997increasing} and  \cite[Theorem 1]{basdevant2015discrete}.
	\begin{corollary}
		With probability $1$, for all $x, y \geq 0$,
		\begin{equation*}
		\lim_{t \to \infty} \frac{H^{\mathsf{d}}(tx, ty)}{t} = 
		\begin{cases}
		(1-p)^{-1} \big(2\sqrt{pxy} - (x+y)p\big) &\qquad \text{ if } p < \min(\tfrac{x}{y}, \tfrac{y}{x}),\\
		\min(x, y) &\qquad \text{else}.
		\end{cases}
		\end{equation*}
	\end{corollary}
	\begin{proof}
		One can readily check that the right-hand side above equals $y - g(x, y)$ when $b_1 = 0$ and $b_2 = 1-p$, using the formula of $g$ in Remark \ref{rmk:limitshape}. 
		The corollary then follows from a combination of Theorems  \ref{thm:maincs6v} and \ref{thm:cs6vdham}. 
	\end{proof}
	\subsection{Connection to the $t$-PNG model}  
The $t$-PNG model introduced in \cite{aggarwal2021deformed} is a deformation of the polynuclear growth (PNG) model \cite{prahofer2000statistical}. We fix a parameter $t \in [0,1]$ and place a Poisson point process with intensity $1$ on the upper-right quadrant representing \emph{nucleations}. We draw lines emanating from each of these nucleations in both the upward and rightward directions until they collide with one another. We call these collision points \emph{intersection points}. Given the Poisson nucleations, we sample the outcomes of the intersection points (lines will either cross or annihilate each other) starting with the intersection point which has the smallest sum of $x$- and $y$- coordinates and moving sequentially outward. At an intersection point, the two lines will cross each other with probability $t$ and will annihilate each other with probability $1-t$, forming a corner. Note that when lines cross, they might generate new intersection points. We refer to Figure \ref{fig:tpng} for a sampling of the $t$-PNG model. 
	\begin{figure}[ht]
			\centering
			\begin{tikzpicture}[scale = 0.6]
			\draw[thick, mintbg, ->] (-0.2, 0) -- (8, 0);
			\draw[thick, mintbg, ->] (0, -0.2) -- (0, 8);
			\node at (4, 0.4) {$\times$};
			\node at (2.5, 7) {$\times$};
			\node at (1, 4) {$\times$};
			\node at (6, 3) {$\times$};
			\node at (1.5, 2) {$\times$};
			\draw (1, 8) -- (1, 4) -- (1.5, 4) -- (1.5, 2) -- (4, 2) -- (4, 0.4) -- (8, 0.4);
			\draw (4, 8) -- (4, 2) -- (8, 2);
			\draw (2.5, 8) -- (2.5, 7) -- (6, 7) -- (6, 3) -- (8, 3);
			\end{tikzpicture}
			\caption{A sampling of the $t$-PNG model, where ``$\times$" are the Poisson nucleations.
			}
			\label{fig:tpng}

	\end{figure}
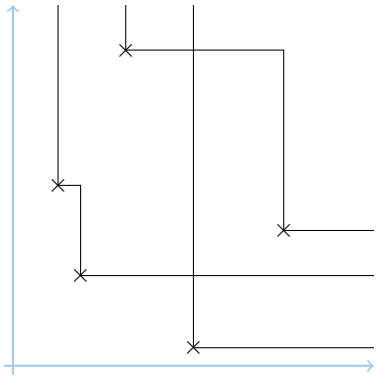

Consider the CS6V model where $b_1$ and $b_2$ do not depend on $i, j$. As noticed by \cite{aggarwal2021deformed}, if we take $b_1 = t$ and $b_2 = 1 - \e^2$ and simultaneously scale the $x$- and $y$-axis by $\e^{-1}$, the CS6V model degenerates to the $t$-PNG model as $\e \to 0$.

	Using integrable methods, \cite{aggarwal2021deformed} proved a weak law of large numbers and a fluctuation theorem for the height function of the $t$-PNG model. In \cite{drillick2022hydrodynamics}, we proved a strong law of large numbers of the $t$-PNG model by constructing a colored $t$-PNG model and using Liggett's superadditive ergodic theorem (one can easily go from the usual subadditive ergodic theorem to the superadditive version by placing negative signs as needed). 
	
	\subsection{Proof ideas}
	For the discrete Hammersley process defined in Section \ref{sec:discreteH}, the strong law of large numbers follows as a direct consequence of the superadditive ergodic theorem (see Theorem \ref{thm:liggettergodic} for the full statement). To apply the superadditive ergodic theorem, one needs to construct a family of superadditive random variables $\{X_{m,n} : 0 \leq m \leq n\}$ where $\{X_{0,n}, n \geq 0\}$ records the height function. A subfamily of the random variables also needs to be ergodic. For the discrete Hammersley process, one can simply define   
\begin{equation}\label{eq:dHXmn}
X_{m, n} := \max\Big\{L: \text{ there exist integer points } (x_1, y_1) \prec\dots \prec (x_L, y_L) \in \xi \cap [m+1, n] \times [m+1, n] \Big\}.
\end{equation}
It is straightforward to check that $\{X_{m, n}: 0 \leq m \leq n\}$ satisfies the condition required by the superadditive ergodic theorem.

For the CS6V model, one can also define a family of random variables $\{X_{m, n}: 0 \leq m \leq n\}$ by slightly modifying the above definition and setting 
\begin{equation*}
X_{m, n} := \max\Big\{L: \text{ there exist integer points } (x_1, y_1) \preceq \dots \preceq (x_L, y_L) \in \eta \cap [m+1, n] \times [m+1, n] \Big\},  
\end{equation*}
where $(x_1, y_1) \preceq (x_2, y_2)$ if $x_1 \leq x_2$ and $y_1 \leq y_2$, and $\eta$ represents the set of vertices of type III and VI in Figure \ref{fig:cs6v}. It is easy to verify that the definition of $X_{m, n}$ above coincides with that in \eqref{eq:dHXmn} when $b_1 = 0$ and $b_2 = 1- p$.

We check whether the conditions in the superadditive ergodic theorem are still satisfied for this modified family of random variables. Indeed, the random variables are still superadditive, and $\{X_{0, n}: n \geq 0\}$ records the height function as desired. However, the ergodic (and even stationary) property no longer holds, since the random variable $X_{m, n}$ depends on the lines entering the box $[m+1, n] \times [m+1, n]$ from the left and bottom boundaries, and the distribution of these lines varies with $m,n$.  
 
To overcome this difficulty, we apply a similar idea as in \cite{drillick2022hydrodynamics} and construct a colored CS6V model. The colored CS6V model is a generalization of the CS6V model, where each line is assigned a color. We denote the different colors by integers $i \in \mathbb{N}$, and we say that the color $i$ has higher priority than the color $j$ if $i < j$. We allow multiple (but only finitely many) lines with different colors to travel together, with the restriction that lines traveling together must have different colors. 

The colored CS6V model is defined by specifying a sampling rule for when horizontal lines and vertical lines meet. The sampling rule is given by a family of stochastic matrices $\{\mathsf{L}^n, n \in \mathbb{N}\}$ that are consistent. More concretely, the matrix $\LL^n$ has both rows and columns indexed by $\{0, 1\}^n \times \{0, 1\}^n$. The  matrix elements are given by $\mathsf{L}^n (\bi, \bj; \bk, \bl)$, where the four vectors $\bi, \bj, \bk, \bl \in \{0, 1\}^n$ specify the number of lines (either zero or one) of each color in $\{1, \dots, n\}$ on the bottom, left, top, and right of an intersection, respectively (see Figure \ref{fig:intersectionconfig}). The stochastic matrices give a probability measure on the output lines $\bk, \bl$ from an intersection point given the input lines $\bi, \bj$. In Section \ref{sec:Lmatrices}, we will explicitly define the matrices $\{\LL^n\}_{n \geq 1}$ and show that they satisfy three properties. We will now give a brief overview of these properties, and the precise statements will be given in Propositions \ref{prop:colorignore}, \ref{prop:mod2erasure}, and \ref{prop:monotonicity}.

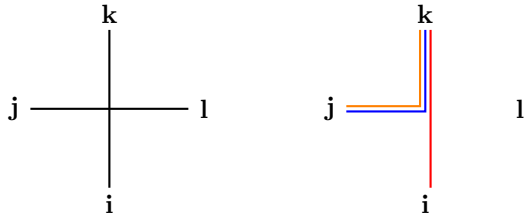
\begin{figure}[ht]
	\centering
	\begin{tikzpicture}[scale = 0.7]
	\begin{scope}
	\draw[thick] (0, 0) -- (3, 0);
	\draw[thick] (1.5, -1.5) -- (1.5, 1.5);
	\node at (1.5, -1.8) {$\bi$}; 
	\node at (-0.3, 0) {$\bj$}; 
	\node at (1.5, 1.8) {$\bk$};
	\node at (3.3, 0) {$\bl$};
	\end{scope}
	\begin{scope}[xshift = 6cm]
	\draw[thick, blue] (0, -0.05) -- (1.5, -0.05) -- (1.5, 1.5);
	\draw[thick, orange] (0, 0.05) -- (1.4, 0.05) -- (1.4, 1.5);
	\draw[thick, red] (1.6, -1.5) -- (1.6, 1.5);
	\node at (1.5, -1.8) {$\bi$}; 
	\node at (-0.3, 0) {$\bj$}; 
	\node at (1.5, 1.8) {$\bk$};
	\node at (3.3, 0) {$\bl$};
	\end{scope}
	\end{tikzpicture}
	\caption{\textbf{Left panel:} Fix $n \in \mathbb{N}$. At an intersection point, we have lines with colors in $\{1, \dots, n\}$ in each direction, with at most one line per color in each direction. Let $\bi, \bj, \bk, \bl \in \{0, 1\}^n$ denote the number of lines on the bottom, left, top and right directions, respectively, where the $m$-th coordinate of each vector records the number of lines with color $m$. \textbf{Right panel:} Take $n = 3$. Let red, blue, and orange denote the colors $1$, $2$, and $3$. We illustrate an example of the configuration with  $\bi = (1, 0, 0)$, $\bj = (0, 1, 1)$,  $\bk = (1, 1, 1)$, and $\bl = (0, 0, 0)$.}
	\label{fig:intersectionconfig}
\end{figure}

	\hypertarget{hyp:property1}{\textbf{Property 1}}
	(Color Ignorance): Lines with higher priority colors ignore those with lower priority colors. For instance, the lines with colors that belong to $\{1, \dots, m\}$ ignore the behavior of lines with colors greater than $m$. This means that if we sample the $n$-colored model and ignore the lines with colors greater than $m$, the remaining lines will behave as the $m$-colored model. Because of this, we can define the random variables $\{X_{m,n}, 0\leq m \leq n\}$ in a way that maintains ergodicity, as we will see in Section \ref{sec:subadditivity}.

	\hypertarget{hyp:property2}{\textbf{Property 2}} (Mod 2 Erasure): Fix arbitrary integers $1 \leq r_1 < \dots < r_m \leq n$. We can project the matrix $\LL^n$ to $\LL^m$ if we replace the colors in $\{r_{k-1}+1, \dots, r_k\}$ with color $k$ for each $k \in\{1, \dots, m\}$ and then erase every pair of lines with the same color traveling together. Because of this property, we can project the colored CS6V model down to the single-colored CS6V model. This will ensure that the random variables $\{X_{0,n}, n \geq 0\}$ will record the height function of the CS6V model, which is the quantity that we are interested in studying.

\hypertarget{hyp:property3}{\textbf{Property 3}} (Monotonicity of the Height Function): Suppose we have a sampling of the two-colored CS6V configuration on a rectangle $[1,x] \times [1,y]$ where we assign the vertex configurations in the rectangle using the matrix $\mathsf{L}^2$, and we allow only second color lines to enter from the left and bottom boundary. We can naturally extend the definition of $H(x,y)$ to apply to the case of non-empty boundary data by defining $H(x,y)$ as the sum of the number of lines entering the box $[1,x] \times [1,y]$ from the left boundary and the number of lines exiting the box $[1,x] \times [1,y]$ from the top boundary. Let $H^1(x,y)$ be the height function when we only consider lines of the first color. Let $H^2(x,y)$ be the height function of the projection of both colors to the single-colored model after we applied the procedure of mod $2$ erasure. Then $$H^1(x,y) \leq H^2(x,y).$$ In other words, adding a second color to the model does not decrease the height function. This property will be important for proving the superadditivity of the random variables $\{X_{m, n}, 0 \leq m \leq n\}$.
\subsubsection{Novelty of the paper} 
In \cite{drillick2022hydrodynamics}, we studied the $t$-PNG model, which is a one-parameter model. The construction of the $\mathsf{L}$-matrices therein 
comes from a sophisticated guess. The proof of stochasticity of $\{\mathsf{L}^n\}_{n \geq 1}$ and Properties \hyperlink{hyp:property1}{1}--\hyperlink{hyp:property3}{3} for the $t$-PNG model is then based on a case by case check. The CS6V model is a two-parameter model, so defining the colored CS6V model via guessing is more challenging. One novelty of the paper is that we use a Boolean-type product to define the $\mathsf{L}$-matrices, which significantly simplifies both the definition and the proof of the properties. This Boolean-type product construction can also be applied to give a simpler definition of the colored $t$-PNG model, see Appendix \ref{sec:Boolean-type product}. 

Other novelties come from working in a discrete setting and the inhomogeneity in the weights. For the $t$-PNG model, due to a scaling property, we only need to prove the strong law of large numbers in the diagonal direction. For the CS6V model, there is no similar scaling property, so we need to prove the strong law of large numbers in every direction. To this end, for each rational direction, we construct a colored model on the first quadrant and prove the strong law of large numbers in that direction. The convergence in all directions then follows from a density argument.

\subsubsection{Another possible approach} We want to remark that it is also possible to prove Theorem \ref{thm:main} without complementing the S6V model. Indeed,  \cite{benassi1987hydrodynamical, andjel2004law} prove a law of large numbers for a family of translation invariant exclusion processes. Their proof also relies on the superadditive ergodic theorem, but in a different way. An important ingredient is to couple the exclusion processes so that we can run them under the same random environment, starting at different times and initial data. The S6V model is also an exclusion-type interacting particle system, and to apply their idea, we need to find such a coupling for the S6V model. Another interesting question is to study the almost sure convergence of the S6V model to the limit shape under general initial data. Note that this result has been proved for a family of finite-range exclusion processes in \cite{bahadoran2010strong}. We leave both of these questions to future work.

\subsection*{Outline} In Section \ref{sec:Lmatrices}, we define a family of $\LL$-matrices using a Boolean-type product. Then we prove that the $\LL$-matrices are stochastic and satisfy Properties \hyperlink{hyp:property1}{1}--\hyperlink{hyp:property3}{3}. In Section \ref{sec:subadditivity}, we construct a colored CS6V model on the first quadrant using the $\LL$-matrices. We apply Liggett's superadditive ergodic theorem to the model and prove Theorem \ref{thm:maincs6v}.  	
	
	\subsection*{Acknowledgments}
	We thank Ivan Corwin for suggesting the question and for helpful comments on the paper. We thank Pablo Ferrari for helpful discussion. HD was supported by the National Science Foundation Graduate Research Fellowship under Grant No. DGE-1644869, as well as the W.M. Keck Foundation Science and Engineering grant on ``Extreme Diffusion" and the Fernholz Foundation's Minerva summer fellows program.
	\section{Definition of the $\mathsf{L}$-matrices}
	\label{sec:Lmatrices}
In this section, we first define the $\LL^1$-matrix, which encodes the weights of the usual single-colored CS6V model. Then, we introduce the Boolean-type product and use it to define the $\LL^n$-matrices. In the following, we take generic parameters $b_1, b_2 \in [0, 1]$.  
	\begin{definition}\label{def:Lmatrix}
		The matrix $\LL^1$ is indexed by a 4-tuple $i, j, k, l \in \{0, 1\}$, where $i, j, k, l$ denote the number of lines (either zero or one) on the bottom, left, top, and right of a vertex. We define
		\begin{align*}
		&\LL^1(1, 0; 1, 0) = 1, \qquad \LL^1(0, 1; 0, 1) = 1, \qquad \LL^1(0, 0; 0, 0) = b_2,\\
		&\LL^1(0, 0; 1, 1) = 1 - b_2, \qquad   \LL^1(1, 1; 1, 1) = b_1, \qquad  \LL^1(1, 1; 0, 0) = 1-b_1.
		\end{align*}
		For all other $i, j, k, l \in \{0, 1\}$, we set $\LL^1(i, j; k, l) = 0$. For fixed input lines $i, j \in \{0, 1\}$, $\LL^1(i, j; \cdot, \cdot)$ is a probability measure on the output lines.
	\end{definition}
	We proceed to give a closed form to the matrices $\{\LL^n\}_{n \geq 1}$. Let 
	\begin{equation*}
	W' := \{0,1,b_1, b_2, 1-b_1, 1-b_2\}
	\end{equation*} 
	denote the set of possible weights for the CS6V model. Note that we treat the elements in $W'$ as indeterminates, so the size of $W'$ does not depend on the values of $b_1$ and $b_2$. We define a product $*$ on the extended space
	\begin{equation}\label{eq:setW}
	W := \{0, 1, b_1, b_2, (1-b_1), (1-b_2), b_1 b_2, b_1 (1 - b_2), (1-b_1) b_2, (1-b_1) (1-b_2)\}.
	\end{equation}
	For $x, y \in W$, we define $[x]$ to be the set of factors of $x$. More precisely, $[x] = \{x\}$ if $x \in W'$, $[x] = \{b_1, b_2\}$ if $x = b_1 b_2$ and similarly for other values. We define 
	\begin{equation*}
	x * y := \begin{cases}
	0 &\qquad \text{ if } \{b_1, 1-b_1\}  \subseteq [x] \cup [y] \text{ or } \{b_2, 1-b_2\} \subseteq [x] \cup [y],\\ 
	\prod_{z \in [x] \cup [y]} z &\qquad \text{ otherwise.} 
	\end{cases}
	\end{equation*}
It is straightforward to check that $W$ is closed under $*$. Moreover, we readily obtain the following. 
	\begin{prop}\label{prop:distributivity}
	The product $*$ satisfies the following three properties.		
	\begin{enumerate}[leftmargin = 2em]
		\item  (commutativity) $a*b = b*a$, for $a, b \in W$.
		\item (associativity) $a * (b * c) = (a*b)*c$, for $a, b, c \in W$.
		\item (distributivity) Suppose that $a_1, ..., a_n \in W$ such that $\sum_{i=1}^n a_i \in W$. For any $c \in W$, 
		\begin{equation*}
		\sum_{i=1}^n c * a_i= c * \sum_{i=1}^n a_i.
		\end{equation*}
	\end{enumerate}	
\end{prop}
Alternatively, we can also define $*$ first on $W'$ and extend it to $W$. For $x, y \in W'$, we define $x * x = x$, and $x * y = 0$ if $\{x, y\} = \{b_1, 1-b_1\}$ or $\{b_2, 1-b_2\}$. This is why we call $*$ a Boolean-type product. For other choices of $x, y \in W'$, we define $x* y = xy$. We then extend the product $*$ from $W'$ to $W$ using the commutative and associative properties.

We proceed to define the $\LL^n$-matrices, using $W$ as the set of possible weights. For $\mathbf{x} = (x_1, \dots, x_n) \in \{0, 1\}^n$ and $r \in \{1, \dots, n\}$, we define the \textbf{$r$-fold projection} 
	\begin{equation*}
	\mfs_r (\mathbf{x}) = \Big(\sum_{m=1}^r x_m\Big) \mod 2.
	\end{equation*}
	As an example, consider the vector $\mathbf{x} = (1,0,1,1).$ We have \begin{align*}
	\mfs_1 (\mathbf{x}) =1, \ \  \mfs_2 (\mathbf{x})  =1, \ \ \mfs_3 (\mathbf{x})=0, \ \ \mfs_4 (\mathbf{x}) = 1.
	\end{align*}
	\begin{definition}\label{def:nLmatrix}
		Fix arbitrary $n \in \Z_{\geq 1}$. The matrix $\LL^n$ is indexed by a 4-tuple $\bi, \bj, \bk, \bl \in \{0, 1\}^n$, where $\bi, \bj, \bk, \bl$ denote the number of lines (either zero or one) of each color in $\{1, \dots, n\}$ on the bottom, left, top, and right of an intersection, respectively. We define the matrix $\LL^n$ via 
		\begin{equation}\label{eq:weight}
		\LL^n(\bi, \bj; \bk, \bl) := \sideset{}{^*}\prod_{r \in \{1, \dots, n\}} \LL^1 (\mfs_r (\bi), \mfs_r (\bj); \mfs_r (\bk), \mfs_r (\bl)),
		\end{equation}
where $\prod^*$ denotes the Boolean-type product of multiple terms under $*$.
	\end{definition}
	\subsection{Properties of $\LL^n$}
	We proceed to state and prove that the matrices $\LL^n$ are stochastic and that they satisfy the properties mentioned above. The key to the proofs of stochasticity, color ignorance, and mod 2 erasure is the distributivity of the Boolean-type product, which allows us to manipulate the sums and products as we normally would. 
	\begin{prop}[Stochasticity] \label{prop:stochasticity}
		$\LL^n$ is a stochastic matrix, i.e., the entries of $\LL^n$ are non-negative, and for any $\bi, \bj \in \{0, 1\}^n$, 
		\begin{equation*}
		\sum_{\bk, \bl \in \{0, 1\}^n} \LL^n(\bi, \bj; \bk, \bl) = 1.
		\end{equation*}
	\end{prop}
	
	\begin{proof}
		We prove this by induction. The case $n=1$ follows from examination. Suppose that $\LL^{n-1}$ is stochastic. Using the definition for $\LL^n$, we can factor it as follows:
		\begin{equation}\label{eq:product}
		\LL^n(\bi, \bj, \bk, \bl) = \LL^{n-1}(\bi_{[1, n-1]}, \bj_{[1,n-1]}, \bk_{[1,n-1]}, \bl_{[1,n-1]}) * \LL^1(\mfs_n (\bi), \mfs_n (\bj); \mfs_n (\bk), \mfs_n (\bl)).
		\end{equation} Note that $(\mfs_n(\bk), \mfs_n (\bl))$ equals each element of $\{0, 1\}^2$ exactly once when we vary $\bk, \bl$ under the restriction $\bk_{[1, n-1]} = \mfk$ and $\bl_{[1, n-1]} = \mfl$. Applying this and Proposition \ref{prop:distributivity} to the right-hand side of \eqref{eq:product} yields
		\begin{align*}
		\sum_{\bk, \bl \in \{0, 1\}^n} \LL^n(\bi, \bj, \bk, \bl)%&= \sum_{\bk, \bl \in \{0, 1\}^n} \LL^{n-1}(\bi_{[1, n-1]}, \bj_{[1,n-1]}; \bk_{[1,n-1]}, \bl_{[1,n-1]}) * \LL^1(\mfs_n (\bi), \mfs_n (\bj); \mfs_n (\bk), \mfs_n (\bl)) \\
		% &= \sum_{\bk, \bl \in \{0, 1\}^{n-1}}L^{n-1}(\bi_{[1, n-1]}, \bj_{[1,n-1]}, \bk, \bl) * \left(\sum_{k,l} L^1(\mfs_n (\bi), \mfs_n (\bj); k,l)\right) \\
		= \Big(\sum_{\mfk, \mfl \in \{0, 1\}^{n-1}} \LL^{n-1}(\bi_{[1, n-1]}, \bj_{[1,n-1]}; \mfk, \mfl)\Big) * \sum_{k,l \in \{0,1\}} \LL^1(\mfs_n (\bi), \mfs_n (\bj); k,l) =1.
		\end{align*}
		The last equality follows from the stochasticity of both $\LL^{n-1}$ and $\LL^1$.  
	\end{proof}
	\begin{prop}[Color Ignorance]\label{prop:colorignore}
		Fix $m \in \{1, \dots, n\}$ and 
		$\mfi, \mfj, \mfk, \mfl \in \{0, 1\}^m$. For all $\bi, \bj \in \{0, 1\}^n$ such that $\bi_{[1, m]} = \mfi$ and $\bj_{[1, m]} = \mfj$, we have 
		\begin{equation}\label{eq:property1}
		\sum_{\substack{\bk_{[1, m]} = \mfk, \\ \bl_{[1, m]} = \mfl}} \LL^n(\bi, \bj; \bk, \bl) = \LL^m(\mfi, \mfj; \mfk, \mfl).
		\end{equation}
	\end{prop}
	
	\begin{proof}
		By Definition \ref{def:nLmatrix}, we have
		\begin{align*}
		\sum_{\substack{\bk_{[1, m]} = \mfk, \\ \bl_{[1, m]} = \mfl}} \LL^n(\bi, \bj; \bk, \bl) = \LL^m(\mfi, \mfj; \mfk, \mfl) \sum_{\substack{\bk_{[1, m]} = \mfk, \\ \bl_{[1, m]} = \mfl}} \ \prods_{i=m+1}^n \LL^1(\mfs_i (\bi), \mfs_i (\bj); \mfs_i (\bk), \mfs_i (\bl)).
		\end{align*}
		The last product can be rewritten as $\LL^{n-m} (i, j, k, \ell)$ where $i = (\mfs_{m+1} (\bi), \cdots, \mfs_{n}(\bi))$, and where $j, k$, and $\ell$ are defined similarly. Summing $\LL^{n-m} (i, j, k, \ell)$ over $\mathbf{k}_{[1, m ]} = \mathfrak{k}$ and $\mathbf{l}_{[1, m]} = \mathfrak{l}$ reduces to summing the same object over $k, \ell \in \{0, 1\}^{n-m}$, which equals $1$ due to stochasticity. This concludes the result.
	\end{proof}
We call  $\pi$ a partition of $\{1, \dots, n\}$ if it takes the form of $$\pi = \big\{\{1, \dots, r_1\}, \{r_1 + 1, \dots, r_2\}, \dots, \{r_{m-1} + 1, \dots, r_m\}\big\}$$ for some $m \leq n$ and  $1 = r_1 < r_2 < \dots < r_m  = n$. We define a projection map $g_\pi: \{0, 1\}^n \to \{0, 1\}^m$ such that 
	\begin{equation*}
	g_{\pi} (x_1, \dots, x_n) = \bigg(\Big(\sum_{i = r_{k-1}+1}^{r_k} x_i\Big) \text{ mod } 2\bigg)_{k = 1}^m.
	\end{equation*}
	We define $\ell(\pi) = m$ to be the length of the partition $\pi$.
	\begin{prop}[Mod 2 Erasure]
		\label{prop:mod2erasure}
		Fix a partition $\pi$ of $\{1, \dots, n\}$ such that $\ell(\pi) = m$. Fix $\mfi, \mfj, \mfk, \mfl \in \{0, 1\}^m$. For all $\bi, \bj \in \{0, 1\}^n$ satisfying $g_{\pi} (\bi) = \mfi$ and $g_{\pi} (\bj) = \mfj$, we have 
		\begin{equation*}
		\sum_{\substack{g_\pi(\bk) = \mfk \\
				g_\pi(\bl) = \mfl}} \LL^n(\bi, \bj; \bk, \bl) = \LL^m(\mfi, \mfj; \mfk, \mfl).
		\end{equation*}
	\end{prop}
	
	\begin{proof}
		Define $A = \{r_1, ..., r_m\}$. We define $i, j$ to be the unique $n-m$ dimensional vectors such that  
		\begin{align*}
		(\mfs_1 (i), \cdots, \mfs_{n-m} (i)) &= (\mfs_{1} (\bi), \cdots, \widehat{\mfs}_{r_1}(\bi), \cdots,  \widehat{\mfs}_{r_m}(\bi), \cdots, \mfs_{n}(\bi)),\\
		(\mfs_1 (j), \cdots, \mfs_{n-m} (j)) &= (\mfs_{1} (\bj), \cdots, \widehat{\mfs}_{r_1}(\bj), \cdots,  \widehat{\mfs}_{r_m}(\bj), \cdots, \mfs_{n}(\bj)), 
		\end{align*}
		where $\widehat{x}$ denotes the removal of $x$ in the vector. We have
		\begin{equation}\label{eq:mod2}
		\sum_{\substack{g_\pi(\bk) = \mfk \\
				g_\pi(\bl) = \mfl}} \LL^n(\bi, \bj; \bk, \bl) =  \sum_{\substack{g_\pi(\bk) = \mfk \\
				g_\pi(\bl) = \mfl}} \prods_{i \in A} \LL^1(\mfs_i (\bi), \mfs_i (\bj); \mfs_i (\bk), \mfs_i (\bl))\prods_{i \notin A, i \leq n} \LL^1(\mfs_i (\bi), \mfs_i (\bj); \mfs_i (\bk), \mfs_i (\bl)). 
		\end{equation}
		%Notice that if 
		Since $g_\pi(\bk) = \mfk$ and $g_\pi(\bl) = \mfl$, we have $\prods_{i \in A} \LL^1(\mfs_i (\bi), \mfs_i (\bj); \mfs_i (\bk), \mfs_i (\bl)) = \LL^m(\mfi, \mfj; \mfk, \mfl)$. Using this and Proposition \ref{prop:distributivity}, the right-hand side of \eqref{eq:mod2} can be rewritten as
		\begin{align*}
		\LL^m(\mfi, \mfj; \mfk, \mfl)  \sum_{\substack{g_\pi(\bk) = \mfk \\
				g_\pi(\bl) = \mfl}}\, \prods_{i \notin A, i \leq n} \LL^1(\mfs_i (\bi), \mfs_i (\bj); \mfs_i (\bk), \mfs_i (\bl)) =   \LL^m(\mfi, \mfj; \mfk, \mfl)  \sum_{k,l \in \{0,1\}^{n-m}}\LL^{n-m}(i,j,k,l) = \LL^m(\mfi, \mfj; \mfk, \mfl).
		\end{align*}
		The first equality is due to \eqref{eq:weight}. The last equality follows from the stochasticity of $\LL^{n-m}.$
	\end{proof}

	The next proposition implies that if we start with a CS6V model on $[1, x] \times [1, y]$ with empty boundary data and wish to add some boundary data (i.e., lines entering the box from the bottom and the left), then we can use the colored model to sample the CS6V model with the boundary data in a way so that the height function at $(x,y)$ cannot decrease with the addition of the boundary data.

	\begin{prop}[Monotonicity of the height function]\label{prop:monotonicity} 
Suppose we have a two-colored CS6V model on $[1, x] \times [1, y]$, where the vertex configurations are assigned using the matrix $\mathsf{L}^2$, and the only lines entering from the bottom and left boundaries are lines of the second color. Let $H^1(x,y)$ denote the height function when we only consider lines of the first color (i.e., the lines of the $1$-fold projection). Let $H^2(x,y)$ denote the height function of the $2$-fold projection of both colors to the single-colored model. Then 
		$$H^1(x,y) \leq H^2(x,y).$$ 
	\end{prop}

\begin{proof}%[Proof of Proposition \ref{prop:monotonicity}]
Consider the rectangle $[1, x] \times [1, y]$.	Let $A_1$ (resp. $B_1$) be the number of first-color lines leaving the top (resp. right) boundary. Let $A_2$ (resp. $B_2$) be the number of second-color lines that leave the same boundary and are accompanied by a first-color line. Let $A_3$ (resp. $B_3$) be the number of second-color lines that leave the same boundary without accompanying first-color lines. Let $A_4$ (resp. $B_4$) be the number of second-color lines entering the left (resp. bottom) boundary. We have
		\begin{align*}
		H^1 (x, y) = A_1 = B_1, \qquad
		H^2 (x, y) = A_1 - A_2 + A_3 + A_4 = B_1  -B_2 + B_3 + B_4.
		\end{align*}
By examining all possible two-color configurations in Figure \ref{fig:twocolor}, we know that for a vertex configuration, the number of second-color lines entering from the left and bottom is no smaller than the number of lines exiting above and to the right. This implies that the number of second-color lines entering the rectangle from the left and bottom boundaries is no smaller than the number of second-color lines exiting the top and right boundaries. Hence, $A_4 + B_4 \geq A_2 + A_3 + B_2 + B_3$. Hence, we have either $A_4 \geq A_2 + A_3$ or $B_4 \geq B_2 +B_3$. Assume without loss of generality that the first inequality holds. Then we have
		\begin{equation*}
		    H^1(x, y) = A_1 \leq A_1 - A_2 + A_3 + A_4 = H^2 (x, y). \qedhere
		\end{equation*}
	\end{proof}
	
	\begin{rmk}
		It is possible to define a colored model in the same way as in Definition \ref{def:nLmatrix}, using ordinary multiplication instead of the Boolean-type product. This will still yield a model with the properties of color ignorance and mod 2 erasure. However, with this definition, the colored model will no longer satisfy the monotonicity of the height function. To see why this is, note that with the Boolean definition, we cannot have the configurations in Figure \ref{fig:nucleation}. This is because these configurations have weights $b_1 * (1-b_1) = 0$ and $b_2 * (1-b_2) = 0$. However, with the ordinary product, these configurations would have non-zero weights. %However, 
  As demonstrated in Figure \ref{fig:badConfig}, these configurations will yield samplings that fail the monotonicity of the height function.  
  
		\begin{figure}[ht]
			\centering
			\begin{tikzpicture}
			\begin{scope}[xshift = 0cm]
			\draw[thick,blue] (0.05, 0.5) -- (0.05, 0.05) -- (0.5, 0.05); \draw[thick,red] (-0.05, 0.5) -- (-0.05, -0.5);
			\draw[thick,red] (-0.5, -0.05) -- (0.5, -0.05);
			\end{scope}
			\begin{scope}[xshift = 2cm]
			\draw[thick,blue] (0.05, 0.5) -- (0.05, 0.05) -- (0.5, 0.05); \draw[thick,red] (-0.05, 0.5) -- (-0.05, -0.05) -- (0.5, -0.05);
			\end{scope}
			\end{tikzpicture}
			\caption{The weights of the above configurations equal $b_1 * (1-b_1) = b_2 * (1-b_2) = 0$ using the Boolean-type product, and equal $b_1 (1-b_1)$, $b_2 (1 - b_2)$ using ordinary multiplication.}
			\label{fig:nucleation}
		\end{figure}
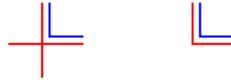

  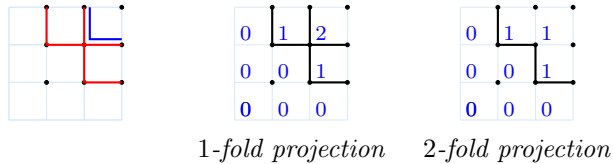
\begin{figure}[ht]
\centering
\begin{tikzpicture} 
\begin{scope}[scale = 0.5]

\foreach \x in {0, 1, 2, 3}
		{\draw[lightmintbg] (\x, 0) -- (\x, 3);
			\draw[lightmintbg] (0, \x) -- (3, \x);
		}

\foreach \x in {0, 1, 2, 3}
		{\draw[lightmintbg] (\x, 0) -- (\x, 3);
			\draw[lightmintbg] (0, \x) -- (3, \x);
		}

  		\foreach \x in {1, 2, 3}
		{
			\foreach \y in {1, 2, 3}
			\draw[fill] (\x, \y) circle (0.05);
		}
% \draw[thick] (0, 3) -- (0, 0) -- (3, 0) -- (3, 3) -- (0, 3);
\draw[thick, red] (1,3) -- (1,2) -- (2,2)--(2,1) -- (3,1);
\draw[thick, red] (2,3) -- (2,2) -- (3,2);
\draw[thick, blue] (2.15,3) -- (2.15,2.15) -- (3,2.15);

\end{scope}

\begin{scope}[xshift =3cm, scale = 0.5]

\foreach \x in {0, 1, 2, 3}
		{\draw[lightmintbg] (\x, 0) -- (\x, 3);
			\draw[lightmintbg] (0, \x) -- (3, \x);
		}

  		\foreach \x in {1, 2, 3}
		{
			\foreach \y in {1, 2, 3}
			\draw[fill] (\x, \y) circle (0.05);
		}
% \draw[thick] (0, 3) -- (0, 0) -- (3, 0) -- (3, 3) -- (0, 3);
\draw[thick, black] (1,3) -- (1,2) -- (2,2)--(2,1) -- (3,1);
\draw[thick, black] (2,3) -- (2,2) -- (3,2);

\node at (1.5, -0.8) {$1$-fold projection};
\foreach \x in {0, 1, 2}
    {
        \node at (\x + 0.3, 0.3) {\footnotesize $\blue{0}$}; 
    }
\foreach \x in {0, 1, 2}
{\node at (0.3, \x + 0.3) {\footnotesize $\blue{0}$}; }
\node at (1.3, 1.3) {\footnotesize $\blue{0}$};
\node at (1.3, 2.3) {\footnotesize $\blue{1}$};
\node at (2.3, 1.3) {\footnotesize $\blue{1}$};
\node at (2.3, 2.3) {\footnotesize $\blue{2}$};
\end{scope}

\begin{scope}[xshift = 6cm, scale = 0.5]
\foreach \x in {0, 1, 2, 3}
		{\draw[lightmintbg] (\x, 0) -- (\x, 3);
			\draw[lightmintbg] (0, \x) -- (3, \x);
		}

  		\foreach \x in {1, 2, 3}
		{
			\foreach \y in {1, 2, 3}
			\draw[fill] (\x, \y) circle (0.05);
		}
% \draw[thick] (0, 3) -- (0, 0) -- (3, 0) -- (3, 3) -- (0, 3);
\draw[thick, black] (1,3) -- (1,2) -- (2,2)--(2,1) -- (3,1);
\node at (1.5, -0.8) {$2$-fold projection};

\foreach \x in {0, 1, 2}
    {
        \node at (\x + 0.3, 0.3) {\footnotesize $\blue{0}$}; 
    }
\foreach \x in {0, 1, 2}
{\node at (0.3, \x + 0.3) {\footnotesize $\blue{0}$}; }
\node at (1.3, 1.3) {\footnotesize $\blue{0}$};
\node at (1.3, 2.3) {\footnotesize $\blue{1}$};
\node at (2.3, 1.3) {\footnotesize $\blue{1}$};
\node at (2.3, 2.3) {\footnotesize $\blue{1}$};

\end{scope} 
\end{tikzpicture}
\caption{A sampling that violates the monotonicity of the height function since the height function in the top-right corner of its $2$-fold projection is smaller than the height function in the top-right corner of its $1$-fold projection.}
\label{fig:badConfig}
\end{figure}

\end{rmk}
	
	\section{The colored model and Proof of Theorem \ref{thm:maincs6v}}
	\label{sec:subadditivity}
With the help of the matrices $\{\mathsf{L}^n\}_{n \in \mathbb{Z}_{\geq 1}}$, we apply the superadditive ergodic theorem to the colored CS6V model and prove the existence of the limit shape. The argument is similar to \cite[Section 3]{drillick2022hydrodynamics} with certain adaptations.

	We are going to construct $\{X_{m, n}: 0 \leq m \leq n\}$ as discussed in the introduction, using the colored CS6V model. Before doing that, let us recall Liggett's superadditive ergodic theorem. For our purposes, we formulate it in the superadditive setting by placing negative signs where necessary. 
	\begin{theorem}[{\cite[page 277]{liggett2012interacting}}]
		\label{thm:liggettergodic}
		Suppose $\{X_{m, n}\}$ is a collection of random variables that is indexed by integers $0 \leq m \leq n$ and satisfies: 
		%Suppose that 
		\begin{enumerate}[leftmargin = 2em, label = (\roman*)]
			\item \label{item:subadditive} Almost surely $X_{0, 0} = 0$ and  $X_{0, n} \geq X_{0, m} + X_{m, n}$ for $0 \leq m \leq n$.
			\item \label{item:ergodic} For each $k \geq 1$, $\{X_{(n-1)k, nk}: n \geq 1\}$ is a stationary and ergodic process. %for each $k \geq 1$. 
			\item \label{item:equalind} $\{X_{m, m+k}: k \geq 0\}  \overset{d}{=} \{X_{m+1, m+k+1}: k \geq 0\}$ %in distribution  
			for each $m \geq 0$.
			\item \label{item:expectation} $\mathbb{E}[X_{0, 1}^-] < \infty$ where $x^{-} = \max(-x, 0)$. %For each $n$, $\mathbb{E}|X_{0, n}| < \infty$ and $\mathbb{E} X_{0, n} \geq -cn$ for some constant $c$.
		\end{enumerate}
		Then there exists %a constant  $\gamma \in (-\infty, \infty]$ and 
		%	a random variable $X_\infty \in (-\infty, \infty]$ and 
		a constant $\gamma = \sup_{n \geq 1} \frac{\mathbb{E}[X_{0, n}]}{n} \in (-\infty, \infty]$ satisfying   
		\begin{align*}
		%\gamma &= \lim_{n \to 
		%\infty} \frac{1}{n} \mathbb{E} %X_{0, n} = \inf_{n \geq 1} \frac{1}{n} \mathbb{E} X_{0, n}, \\
		\gamma = \lim_{n \to \infty} \frac{X_{0, n}}{n} \text{ a.s.}
		%&\gamma = \sup_{n \geq 1} \frac{\mathbb{E}[X_{0, n}]}{n} \in (-\infty, \infty]
		\end{align*}
	\end{theorem}
We proceed to define a colored CS6V model on the first quadrant. Recall that $I, J$ are the periodicities of the weights of the CS6V model, as in Assumption \ref{assumption}. Fix a direction $(x, y) \in \mathbb{Q}_{> 0}^2$, and define 
	\begin{equation}\label{eq:defN}
	N:= \text{ the smallest positive integer such that } \frac{Nx}{I}, \frac{Ny}{J} \in \mathbb{N}. 
	\end{equation}
	We have omitted the dependence of $N$ on $x, y$. Assign color $-k$ to the vertex $(a, b)$ if either $a \in [N(k-1)x+1, Nkx]$ or $b \in [N(k-1)y+1, Nky]$, see the left panel of Figure \ref{fig:step}. We use negative numbers to label the colors so that we can have infinitely many colors of increasing priority. Given the input lines of a vertex with color $-k$, we sample the output lines according to stochastic matrices $\LL^k$ with parameters $b_1 = b_1 (i, j)$ and $b_2 = b_2 (i, j)$. The colored CS6V model is then sampled sequentially at the vertices $(1, 1), (2, 1), (1, 2), (3, 1), (2, 2), (1, 3), \dots$ to the entire quadrant, using the $\mathsf{L}$-matrices assigned to different vertices. By definition, $\LL^k$ is defined for the positive colors $1, \ldots, k$, so we just map the colors $-k, \ldots, -1$ to $1, \ldots, k$, preserving their order. One can check that if a vertex of color $-k$ has no input lines, then the only possible output lines emanating from that vertex can be of color $-k$. See the right panel of Figure \ref{fig:step} for a sampling of the colored CS6V model. 

	We proceed to define the random variables $\{X_{m, n}, m, n \in \Z_{\geq 0}, m \leq n\}$ that satisfy the condition of Theorem \ref{thm:liggettergodic}. Let $v_{z, Nny}$ denote the number of vertical lines exiting the vertex $(z, Nny)$ with colors that belong to $\{-n, \dots, -m-1\}$. We define 
	\begin{equation}\label{eq:Xmn}
	X_{m, n} := \sum_{z = Nmx+1}^{Nnx} \big(v_{z, Nny} \text{ mod } 2\big).
	\end{equation}
	The next proposition encodes the height function of the CS6V model in $\{X_{m, n}: 0 \leq m \leq n\}$. 
	
		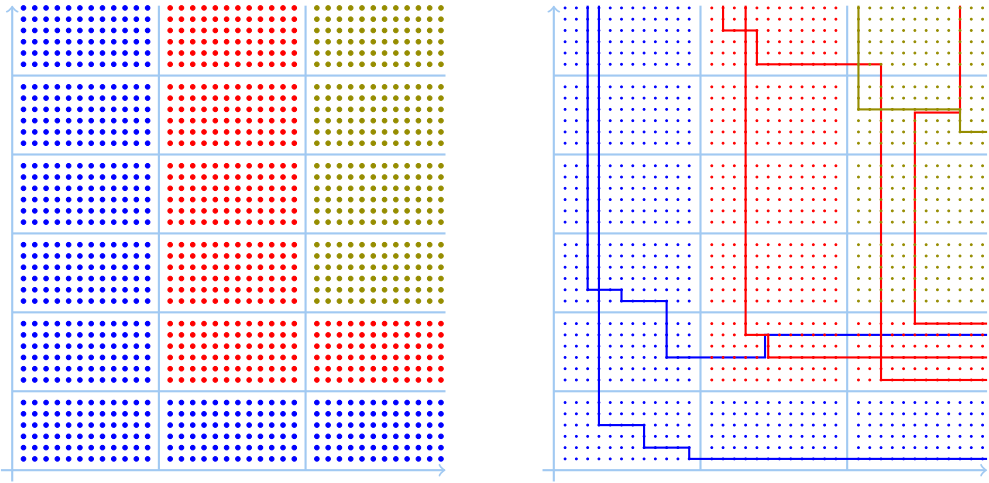
\begin{figure}[ht]
		\centering
		\begin{tikzpicture}[scale = 0.6]
		
	\begin{scope}[xshift = 0cm]
		
		\draw[thick,  mintbg, ->] (-.25, 0) -- (9.6, 0);
        \draw[thick,  mintbg, ->] (0, -.25) -- (0, 10.3);
        \draw[thick,  mintbg] (3.25, 0) -- (3.25, 10.3);
        \draw[thick,  mintbg] (6.5, 0) -- (6.5, 10.3);
        \draw[thick,  mintbg] (0, 1.75) -- (9.6, 1.75);
        \draw[thick,  mintbg] (0, 3.5) -- (9.6, 3.5);
        \draw[thick,  mintbg] (0, 5.25) -- (9.6, 5.25);
        \draw[thick,  mintbg] (0, 7) -- (9.6, 7);
        \draw[thick,  mintbg] (0, 8.75) -- (9.6, 8.75);

        \foreach \x in {1, 2, 3, 4, 5, 6, 7, 8, 9, 10, 11, 12}
		{
			\foreach \y in {1, 2, 3, 4, 5, 6}
			\draw[fill, blue] (\x/4,  \y/4) circle (0.05);
		}
		
		\foreach \x in {1, 2, 3, 4, 5, 6, 7, 8, 9, 10, 11, 12}
		{
			\foreach \y in {1, 2, 3, 4, 5, 6}
			\draw[fill, blue] (3.25+ \x/4, \y/4) circle (0.05);
		}
		
			\foreach \x in {1, 2, 3, 4, 5, 6, 7, 8, 9, 10, 11, 12}
		{
			\foreach \y in {1, 2, 3, 4, 5, 6}
			\draw[fill, blue] (6.5 + \x/4,  \y/4) circle (0.05);
		}
	%%%%%%%%%%%%%%%%%%%%%%%%%%%%%%%%%%%%%%%%%%%%%%%%%%%%%%%%%%%%%%%%%
	
	        \foreach \x in {1, 2, 3, 4, 5, 6, 7, 8, 9, 10, 11, 12}
		{
			\foreach \y in {1, 2, 3, 4, 5, 6}
			\draw[fill, blue] (\x/4,  1.75 + \y/4) circle (0.05);
		}
		
		\foreach \x in {1, 2, 3, 4, 5, 6, 7, 8, 9, 10, 11, 12}
		{
			\foreach \y in {1, 2, 3, 4, 5, 6}
			\draw[fill, red] (3.25+ \x/4, 1.75 +\y/4) circle (0.05);
		}
		
			\foreach \x in {1, 2, 3, 4, 5, 6, 7, 8, 9, 10, 11, 12}
		{
			\foreach \y in {1, 2, 3, 4, 5, 6}
			\draw[fill, red] (6.5 + \x/4,  1.75+ \y/4) circle (0.05);
		}
	%%%%%%%%%%%%%%%%%%%%%%%%%%%%%%%%%%%%%%%%%%%%%%%%%%%%%%%%%%%%%%%%%
	
		%%%%%%%%%%%%%%%%%%%%%%%%%%%%%%%%%%%%%%%%%%%%%%%%%%%%%%%%%%%%%%%%%
	
	        \foreach \x in {1, 2, 3, 4, 5, 6, 7, 8, 9, 10, 11, 12}
		{
			\foreach \y in {1, 2, 3, 4, 5, 6}
			\draw[fill, blue] (\x/4,  3.5 + \y/4) circle (0.05);
		}
		
		\foreach \x in {1, 2, 3, 4, 5, 6, 7, 8, 9, 10, 11, 12}
		{
			\foreach \y in {1, 2, 3, 4, 5, 6}
			\draw[fill, red] (3.25+ \x/4, 3.5 +\y/4) circle (0.05);
		}
		
			\foreach \x in {1, 2, 3, 4, 5, 6, 7, 8, 9, 10, 11, 12}
		{
			\foreach \y in {1, 2, 3, 4, 5, 6}
			\draw[fill, olive] (6.5 + \x/4,  3.5+ \y/4) circle (0.05);
		}
	%%%%%%%%%%%%%%%%%%%%%%%%%%%%%%%%%%%%%%%%%%%%%%%%%%%%%%%%%%%%%%%%%
	
			%%%%%%%%%%%%%%%%%%%%%%%%%%%%%%%%%%%%%%%%%%%%%%%%%%%%%%%%%%%%%%%%%
	
	        \foreach \x in {1, 2, 3, 4, 5, 6, 7, 8, 9, 10, 11, 12}
		{
			\foreach \y in {1, 2, 3, 4, 5, 6}
			\draw[fill, blue] (\x/4,  5.25 + \y/4) circle (0.05);
		}
		
		\foreach \x in {1, 2, 3, 4, 5, 6, 7, 8, 9, 10, 11, 12}
		{
			\foreach \y in {1, 2, 3, 4, 5, 6}
			\draw[fill, red] (3.25+ \x/4, 5.25 +\y/4) circle (0.05);
		}
		
			\foreach \x in {1, 2, 3, 4, 5, 6, 7, 8, 9, 10, 11, 12}
		{
			\foreach \y in {1, 2, 3, 4, 5, 6}
			\draw[fill, olive] (6.5 + \x/4,  5.25+ \y/4) circle (0.05);
		}
	%%%%%%%%%%%%%%%%%%%%%%%%%%%%%%%%%%%%%%%%%%%%%%%%%%%%%%%%%%%%%%%%%
				%%%%%%%%%%%%%%%%%%%%%%%%%%%%%%%%%%%%%%%%%%%%%%%%%%%%%%%%%%%%%%%%%
	
	        \foreach \x in {1, 2, 3, 4, 5, 6, 7, 8, 9, 10, 11, 12}
		{
			\foreach \y in {1, 2, 3, 4, 5, 6}
			\draw[fill, blue] (\x/4,  7 + \y/4) circle (0.05);
		}
		
		\foreach \x in {1, 2, 3, 4, 5, 6, 7, 8, 9, 10, 11, 12}
		{
			\foreach \y in {1, 2, 3, 4, 5, 6}
			\draw[fill, red] (3.25+ \x/4, 7 +\y/4) circle (0.05);
		}
		
			\foreach \x in {1, 2, 3, 4, 5, 6, 7, 8, 9, 10, 11, 12}
		{
			\foreach \y in {1, 2, 3, 4, 5, 6}
			\draw[fill, olive] (6.5 + \x/4,  7+ \y/4) circle (0.05);
		}
	%%%%%%%%%%%%%%%%%%%%%%%%%%%%%%%%%%%%%%%%%%%%%%%%%%%%%%%%%%%%%%%%%
				%%%%%%%%%%%%%%%%%%%%%%%%%%%%%%%%%%%%%%%%%%%%%%%%%%%%%%%%%%%%%%%%%
	
	        \foreach \x in {1, 2, 3, 4, 5, 6, 7, 8, 9, 10, 11, 12}
		{
			\foreach \y in {1, 2, 3, 4, 5, 6}
			\draw[fill, blue] (\x/4,  8.75 + \y/4) circle (0.05);
		}
		
		\foreach \x in {1, 2, 3, 4, 5, 6, 7, 8, 9, 10, 11, 12}
		{
			\foreach \y in {1, 2, 3, 4, 5, 6}
			\draw[fill, red] (3.25+ \x/4, 8.75 +\y/4) circle (0.05);
		}
		
			\foreach \x in {1, 2, 3, 4, 5, 6, 7, 8, 9, 10, 11, 12}
		{
			\foreach \y in {1, 2, 3, 4, 5, 6}
			\draw[fill, olive] (6.5 + \x/4,  8.75+ \y/4) circle (0.05);
		}
	%%%%%%%%%%%%%%%%%%%%%%%%%%%%%%%%%%%%%%%%%%%%%%%%%%%%%%%%%%%%%%%%%

\end{scope}

\begin{scope}[xshift = 12cm]

		\draw[thick,  mintbg, ->] (-.25, 0) -- (9.6, 0);
        \draw[thick,  mintbg, ->] (0, -.25) -- (0, 10.3);
        \draw[thick,  mintbg] (3.25, 0) -- (3.25, 10.3);
        \draw[thick,  mintbg] (6.5, 0) -- (6.5, 10.3);
        \draw[thick,  mintbg] (0, 1.75) -- (9.6, 1.75);
        \draw[thick,  mintbg] (0, 3.5) -- (9.6, 3.5);
        \draw[thick,  mintbg] (0, 5.25) -- (9.6, 5.25);
        \draw[thick,  mintbg] (0, 7) -- (9.6, 7);
        \draw[thick,  mintbg] (0, 8.75) -- (9.6, 8.75);
        
        \draw[thick, blue] (0.75, 10.3) -- (0.75, 4) -- (1, 4) -- (1,1) -- (2,1) -- (2,.5) -- (3,0.5) -- (3, 0.25) --(9.6, .25);
        \draw[thick, blue] (1, 10.3) -- (1, 4) -- (1.5,4) -- (1.5, 3.75) -- (2.5, 3.75) -- (2.5, 2.5) -- (4.25, 2.5)  -- (4.68, 2.5) -- (4.68, 3.0)--(9.6, 3.0);
        \draw[thick, red] (4.25, 10.3) -- (4.25, 3)--(4.75, 3) -- (4.75, 2.5) -- (9.6, 2.5);
        \draw[thick, red] (3.75, 10.3) -- (3.75, 9.75)--(4.5, 9.75) -- (4.5, 9) --(7.25, 9) --(7.25, 2)-- (9.6, 2);
        \draw[thick, red] (9.6, 3.25) --(8, 3.25) --(8,7.93) --(9, 7.93)--(9, 10.3); 
        \draw[thick, olive] (6.75, 10.3) -- (6.75, 8)--(9, 8)--(9, 7.5) --(9.6, 7.5);

        \foreach \x in {1, 2, 3, 4, 5, 6, 7, 8, 9, 10, 11, 12}
		{
			\foreach \y in {1, 2, 3, 4, 5, 6}
			\draw[fill, blue] (\x/4,  \y/4) circle (0.02);
		}
		
		\foreach \x in {1, 2, 3, 4, 5, 6, 7, 8, 9, 10, 11, 12}
		{
			\foreach \y in {1, 2, 3, 4, 5, 6}
			\draw[fill, blue] (3.25+ \x/4, \y/4) circle (0.02);
		}
		
			\foreach \x in {1, 2, 3, 4, 5, 6, 7, 8, 9, 10, 11, 12}
		{
			\foreach \y in {1, 2, 3, 4, 5, 6}
			\draw[fill, blue] (6.5 + \x/4,  \y/4) circle (0.02);
		}
	%%%%%%%%%%%%%%%%%%%%%%%%%%%%%%%%%%%%%%%%%%%%%%%%%%%%%%%%%%%%%%%%%
	
	        \foreach \x in {1, 2, 3, 4, 5, 6, 7, 8, 9, 10, 11, 12}
		{
			\foreach \y in {1, 2, 3, 4, 5, 6}
			\draw[fill, blue] (\x/4,  1.75 + \y/4) circle (0.02);
		}
		
		\foreach \x in {1, 2, 3, 4, 5, 6, 7, 8, 9, 10, 11, 12}
		{
			\foreach \y in {1, 2, 3, 4, 5, 6}
			\draw[fill, red] (3.25+ \x/4, 1.75 +\y/4) circle (0.02);
		}
		
			\foreach \x in {1, 2, 3, 4, 5, 6, 7, 8, 9, 10, 11, 12}
		{
			\foreach \y in {1, 2, 3, 4, 5, 6}
			\draw[fill, red] (6.5 + \x/4,  1.75+ \y/4) circle (0.02);
		}
	%%%%%%%%%%%%%%%%%%%%%%%%%%%%%%%%%%%%%%%%%%%%%%%%%%%%%%%%%%%%%%%%%
	
		%%%%%%%%%%%%%%%%%%%%%%%%%%%%%%%%%%%%%%%%%%%%%%%%%%%%%%%%%%%%%%%%%
	
	        \foreach \x in {1, 2, 3, 4, 5, 6, 7, 8, 9, 10, 11, 12}
		{
			\foreach \y in {1, 2, 3, 4, 5, 6}
			\draw[fill, blue] (\x/4,  3.5 + \y/4) circle (0.02);
		}
		
		\foreach \x in {1, 2, 3, 4, 5, 6, 7, 8, 9, 10, 11, 12}
		{
			\foreach \y in {1, 2, 3, 4, 5, 6}
			\draw[fill, red] (3.25+ \x/4, 3.5 +\y/4) circle (0.02);
		}
		
			\foreach \x in {1, 2, 3, 4, 5, 6, 7, 8, 9, 10, 11, 12}
		{
			\foreach \y in {1, 2, 3, 4, 5, 6}
			\draw[fill, olive] (6.5 + \x/4,  3.5+ \y/4) circle (0.02);
		}
	%%%%%%%%%%%%%%%%%%%%%%%%%%%%%%%%%%%%%%%%%%%%%%%%%%%%%%%%%%%%%%%%%
	
			%%%%%%%%%%%%%%%%%%%%%%%%%%%%%%%%%%%%%%%%%%%%%%%%%%%%%%%%%%%%%%%%%
	
	        \foreach \x in {1, 2, 3, 4, 5, 6, 7, 8, 9, 10, 11, 12}
		{
			\foreach \y in {1, 2, 3, 4, 5, 6}
			\draw[fill, blue] (\x/4,  5.25 + \y/4) circle (0.02);
		}
		
		\foreach \x in {1, 2, 3, 4, 5, 6, 7, 8, 9, 10, 11, 12}
		{
			\foreach \y in {1, 2, 3, 4, 5, 6}
			\draw[fill, red] (3.25+ \x/4, 5.25 +\y/4) circle (0.02);
		}
		
			\foreach \x in {1, 2, 3, 4, 5, 6, 7, 8, 9, 10, 11, 12}
		{
			\foreach \y in {1, 2, 3, 4, 5, 6}
			\draw[fill, olive] (6.5 + \x/4,  5.25+ \y/4) circle (0.02);
		}
	%%%%%%%%%%%%%%%%%%%%%%%%%%%%%%%%%%%%%%%%%%%%%%%%%%%%%%%%%%%%%%%%%
				%%%%%%%%%%%%%%%%%%%%%%%%%%%%%%%%%%%%%%%%%%%%%%%%%%%%%%%%%%%%%%%%%
	
	        \foreach \x in {1, 2, 3, 4, 5, 6, 7, 8, 9, 10, 11, 12}
		{
			\foreach \y in {1, 2, 3, 4, 5, 6}
			\draw[fill, blue] (\x/4,  7 + \y/4) circle (0.02);
		}
		
		\foreach \x in {1, 2, 3, 4, 5, 6, 7, 8, 9, 10, 11, 12}
		{
			\foreach \y in {1, 2, 3, 4, 5, 6}
			\draw[fill, red] (3.25+ \x/4, 7 +\y/4) circle (0.02);
		}
		
			\foreach \x in {1, 2, 3, 4, 5, 6, 7, 8, 9, 10, 11, 12}
		{
			\foreach \y in {1, 2, 3, 4, 5, 6}
			\draw[fill, olive] (6.5 + \x/4,  7+ \y/4) circle (0.02);
		}
	%%%%%%%%%%%%%%%%%%%%%%%%%%%%%%%%%%%%%%%%%%%%%%%%%%%%%%%%%%%%%%%%%
				%%%%%%%%%%%%%%%%%%%%%%%%%%%%%%%%%%%%%%%%%%%%%%%%%%%%%%%%%%%%%%%%%
	
	        \foreach \x in {1, 2, 3, 4, 5, 6, 7, 8, 9, 10, 11, 12}
		{
			\foreach \y in {1, 2, 3, 4, 5, 6}
			\draw[fill, blue] (\x/4,  8.75 + \y/4) circle (0.02);
		}
		
		\foreach \x in {1, 2, 3, 4, 5, 6, 7, 8, 9, 10, 11, 12}
		{
			\foreach \y in {1, 2, 3, 4, 5, 6}
			\draw[fill, red] (3.25+ \x/4, 8.75 +\y/4) circle (0.02);
		}
		
			\foreach \x in {1, 2, 3, 4, 5, 6, 7, 8, 9, 10, 11, 12}
		{
			\foreach \y in {1, 2, 3, 4, 5, 6}
			\draw[fill, olive] (6.5 + \x/4,  8.75+ \y/4) circle (0.02);
		}
	%%%%%%%%%%%%%%%%%%%%%%%%%%%%%%%%%%%%%%%%%%%%%%%%%%%%%%%%%%%%%%%%%

\end{scope}

		\end{tikzpicture}
		\caption{\textbf{Left panel:} We illustrate how we color the vertices %in the CS6V model
on the first quadrant when $I = 3, J = 2, x = 2$, and $y = 1$. In this case, $N = 6$. Let blue, red, and olive represent the colors $-1, -2$, and $-3$. Each rectangle contains $12 \times 6$ vertices with the same color. \textbf{Right panel:} A sampling of the colored CS6V model. The lines emanating from a vertex inherit the color of that vertex. Due to the property of color ignorance, the red lines can be sampled ignoring the blue lines, and the olive lines can be sampled ignoring both the blue and red lines. Note that lines of different colors can travel together across the same vertices.}
\label{fig:step}
	\end{figure}

	\begin{prop}\label{prop:equalind}
		We have $\{X_{0, k}, k \in \mathbb{Z}_{\geq 0}\} = \{H(Nkx, Nky), k \in \mathbb{Z}_{\geq 0}\}$. %where $H$ is the height function of the inhomogeneous S6V model with periodicity $I, J$. 
	\end{prop}
	
	\begin{proof}
		It suffices to show that for each $n \in \mathbb{N}$, $\{X_{0, k}, k = 0, \dots, n\} \overset{d}{=} \{H(Nkx, Nky), k = 0, \dots, n\}$. The lines in the square $[\frac{1}{2}, Nnx + \frac{1}{2}] \times [\frac{1}{2}, Nny +\frac{1}{2}]$ have colors that belong to $\{-n, \dots, -1\}$. Replace these colors with a single color. Then, by the property of mod 2 erasure, the resulting model reduces to the single-colored CS6V model. Note that $X_{0, k}$ in \eqref{eq:Xmn} is equal to the number of vertical lines that cross the segment $[\frac{1}{2}, Nkx+\frac{1}{2} ] \times \{Nky+\frac{1}{2}\}$, which is $H(Nkx, Nky)$. This implies that  $\{X_{0, k}, k = 0, \dots, n\} \overset{d}{=} \{H(Nkx, Nky), k = 0, \dots, n\}$.
	\end{proof}

	\begin{prop}\label{prop:item24}
		The stochastic process $\{X_{m, n}: m, n \in \mathbb{Z}_{\geq 0}, m \leq n\}$ satisfies conditions \ref{item:ergodic}--\ref{item:expectation} of Theorem \ref{thm:liggettergodic}.
	\end{prop}
	\begin{proof}
		We first prove \ref{item:ergodic}. Consider the square $[Nmx+\frac{1}{2}, Nnx+\frac{1}{2}] \times [Nmy+\frac{1}{2}, Nny +\frac{1}{2}]$. There are lines flowing inside through the left boundary $\{Nmx + \frac{1}{2}\} \times [Nmy+\frac{1}{2}, Nny+\frac{1}{2}]$ and the bottom boundary $[Nmx+\frac{1}{2}, Nnx+\frac{1}{2}] \times \{Nmy+\frac{1}{2}\}$. These lines have colors belonging to $\{-m, \dots, -1\}$. The vertices in $[Nmx+\frac{1}{2}, Nnx+\frac{1}{2}] \times [Nmy+\frac{1}{2}, Nny+\frac{1}{2}]$ have colors in $\{-n, \dots, -m-1\}$. Note that the color $i$ takes priority over $j$ if $i < j$, so the lines that emanate from the vertices in $[Nmx+\frac{1}{2}, Nnx+\frac{1}{2}] \times [Nmy+\frac{1}{2}, Nny+\frac{1}{2}]$ have higher priority than the lines entering through the left and bottom boundaries. Therefore, by Proposition \ref{prop:colorignore}, the behavior of the lines with colors $\{-n, \dots, -m - 1\}$ in the rectangle $[Nmx+\frac{1}{2}, Nnx+\frac{1}{2}] \times [Nmy+\frac{1}{2},Nny+\frac{1}{2}]$ does not depend on the lower priority lines entering from the left and bottom. Hence, the distribution of $X_{m, n}$ is independent of the number and location of the lines entering the bottom and left boundaries of $[Nmx+\frac{1}{2}, Nnx+\frac{1}{2}] \times [Nmy+\frac{1}{2}, Nny+\frac{1}{2}]$. This implies that for each $k \geq 1$, the random variables $\{X_{(n-1)k, nk}, n \geq 1\}$ are independent. It is straightforward to see that $X_{(n-1)k, nk}$ has the same distribution as $H(Nkx, Nky)$ for all $n \geq 1$, therefore this sequence is i.i.d, and hence it is stationary and ergodic.

		We proceed to prove \ref{item:equalind}. It suffices to show that for arbitrary $m \in \Z_{\geq 0}$, 
		\begin{equation}\label{eq:equalind}
		\{X_{m, m+k}, k \geq 0\} \overset{d}{=} \{X_{0, k}, k \geq 0\}.
		\end{equation}
		We look at the colored CS6V model restricted to $[mNx+\frac{1}{2}, \infty) \times [mNy+\frac{1}{2}, \infty)$. Note that there are lines with colors in $\{-m, \dots, -1\}$ entering from the left and bottom boundaries of $[mNx+\frac{1}{2}, \infty) \times [mNy+\frac{1}{2}, \infty)$. By Proposition \ref{prop:colorignore}, the behavior of lines in $[mNx+\frac{1}{2}, \infty) \times [mNy+\frac{1}{2}, \infty)$ with colors less than $-m$  is not affected by the lower priority lines entering from the boundary. Since $b_1 (mNx + \cdot, mNy + \cdot) = b_1(\cdot, \cdot)$ and $b_2 (mNx + \cdot, mNy + \cdot) = b_2(\cdot, \cdot)$, this implies that after a shift by $(mNx, mNy)$, the lines with colors $i_1, \dots, i_k \in \Z_{\leq -m-1}$ in $[mNx+\frac{1}{2}, \infty) \times [mNy+\frac{1}{2}, \infty)$ behave the same (in distribution) as the lines with colors $i_1 + m, \dots, i_k + m$ in $[\frac{1}{2}, \infty) \times [\frac{1}{2}, \infty)$. Hence, we conclude \eqref{eq:equalind}.  

		Finally, we have $X_{0, 1}^{-} = 0$ because $X_{0, 1}$ is non-negative. Hence, \ref{item:expectation} holds. 
	\end{proof}

	Let us proceed to prove that $\{X_{m, n}, m, n \in \Z_{\geq 0}, m \leq n\}$ also satisfies the superadditive condition \ref{item:subadditive} in Theorem \ref{thm:liggettergodic}. We begin with some preparation. In the square $[\frac{1}{2}, Nnx+\frac{1}{2}] \times [\frac{1}{2}, Nny+\frac{1}{2}]$, we replace the colors $\{-m, \dots, -1\}$ with the color $-1$ and replace the colors $\{-n, \dots, -m-1\}$ with the color $-2$. After that, as long as there are two lines with the same color that travel together, we erase them.
	By Proposition \ref{prop:mod2erasure}, the resulting model is a two-colored CS6V model. In particular, the vertices have color $-1$ in the L-shaped area $[\frac{1}{2}, Nmx+\frac{1}{2}] \times [\frac{1}{2}, Nny+\frac{1}{2}] \cup [\frac{1}{2}, Nnx+\frac{1}{2}] \times [\frac{1}{2}, Nmy+\frac{1}{2}]$. The vertices have color $-2$ in the square $[Nmx+\frac{1}{2}, Nnx+\frac{1}{2}] \times [Nmy + \frac{1}{2}, Nny+\frac{1}{2}]$.   

	For the resulting two-colored CS6V model, let $Q_1$ be the number of vertical lines with color $-1$ that cross the segment $[\frac{1}{2}, Nmx+\frac{1}{2}] \times \{Nmy+\frac{1}{2}\}$, let $Q_2$ be the number of vertical lines with color $-2$ that cross the segment $[Nmx + \frac{1}{2}, Nnx+\frac{1}{2}] \times \{Nny + \frac{1}{2}\}$, and let $P_1$ be the number of horizontal lines with color $-1$ that cross $\{Nmx + \frac{1}{2}\} \times [Nmy+\frac{1}{2}, Nny+ \frac{1}{2}]$.  Finally, let $R$ be the number of single vertical lines of either color %with a single color $-2$ 
	that cross $[Nmx+\frac{1}{2}, Nnx+\frac{1}{2}] \times \{Nny + \frac{1}{2}\}$ (i.e., lines that do not travel in a pair).
	%let $Q_{1, 2}$ be the number of pairs of vertical lines of colors $-1$ and $-2$ that travel together and cross $[Nmx+\frac{1}{2}, Nnx+\frac{1}{2}] \times \{Nny + \frac{1}{2}\}$.
	
	\begin{lemma}\label{lem:someequality}
		The following result holds:
		\begin{align}
		\label{eq:u0m}
		&
		X_{0, m} = Q_1,
		\\
		\label{eq:umn}
		&X_{m, n}= Q_2,  \\%u^{[0, m]},\\
		\label{eq:u0n}
		&X_{0, n} = Q_1 + P_1 + R.
		%u^{[0, m]} + h^{[m, n]} + u^{[m, n]} - e^{[m, n]}.
		\end{align} 
		%$X_{0, n} = u^{[0, m]} + h^{[m, n]} = X_{0, m} + h^{[m, n]}$ + $X_{m, n} - e^{[m, n]}$.
	\end{lemma}
	\begin{proof}
		Recall that we obtain the two-colored CS6V model in $[\frac{1}{2}, Nnx+\frac{1}{2}] \times [\frac{1}{2}, Nny+\frac{1}{2}]$ by replacing the colors $\{-n, \dots, -m-1\}$ with the color $-2$, replacing the colors $\{-m, \dots, -1\}$ with the color $-1$, and erasing pairs of lines with the same color. The erasure corresponds to the mod 2 erasure procedure in \eqref{eq:Xmn}. Hence, 
		equations \eqref{eq:u0m} and \eqref{eq:umn} directly follow from \eqref{eq:Xmn}.

		We proceed to prove \eqref{eq:u0n}. Note that $X_{0, n}$ is the number of single vertical lines that cross the segment  $[\frac{1}{2}, Nnx+\frac{1}{2}] \times \{Nny + \frac{1}{2}\}$ in the two-colored CS6V model. We decompose
		\begin{equation}\label{eq:Xdecomp}
		X_{0, n} = Y_1 + R,
		%Y_{0, m} + Y_{m, n},
		\end{equation}
		where $Y_1$
		%$Y_{0, m}$ 
		equals the number of vertical lines with the color $-1$ that leave $[\frac{1}{2}, Nmx + \frac{1}{2}] \times \{Nny+\frac{1}{2}\}$. %and %$Y_{m, n}$ 
		We turn to the rectangle $[\frac{1}{2}, Nmx + \frac{1}{2}] \times [Nmy + \frac{1}{2}, Nny+\frac{1}{2}]$, in which there are only lines with color $-1$. The CS6V model is conservative in the sense that the number of lines on the right and bottom of a vertex equals the number of lines on the top and left of that vertex. Therefore, the number of lines that cross the bottom and right boundaries of the rectangle $[\frac{1}{2}, Nmx+\frac{1}{2}] \times [Nmy + \frac{1}{2}, Nny + \frac{1}{2}]$ is equal to the number of lines that cross the top and left boundaries, hence,
		\begin{equation}\label{eq:y1}
		Y_1 = Q_1 + P_1.    
		\end{equation}
	    Using this together with \eqref{eq:Xdecomp} and \eqref{eq:y1}, we conclude \eqref{eq:u0n}.
	\end{proof}
	
	\begin{prop}\label{prop:item1}
		We have	$X_{0, 0} = 0$ and $X_{0, n} \geq X_{0, m} + X_{m, n}$ for $0 \leq m \leq n$. Hence, $\{X_{m, n}, m \leq n \in \Z_{\geq 0}\}$ satisfies \ref{item:subadditive} of Theorem \ref{thm:liggettergodic}.
	\end{prop}
	\begin{proof}
		By definition, we have $X_{0, 0} = 0$. We proceed to show that $X_{0, n} \geq X_{0, m} + X_{m, n}$. By Lemma \ref{lem:someequality}, this is equivalent to $R + P_1 \geq Q_2$. This follows from applying Proposition \ref{prop:monotonicity} to the rectangle $[Nmx+\frac{1}{2}, Nnx+\frac{1}{2}] \times [Nmy+\frac{1}{2}, Nny+\frac{1}{2}]$, where only lower priority lines enter from the bottom-left boundary.  
	\end{proof}
	
	\begin{prop}\label{prop:onedirection}
		Fix arbitrary $x, y \in \mathbb{Q}_{\geq 0}$, there exists a constant $\widetilde{g}(x, y) \in \mathbb{R}_{\geq 0}$ such that almost surely, 
		\begin{equation*}
		\lim_{n \to \infty} \frac{H(\lfloor nx \rfloor, \lfloor ny \rfloor)}{n} = \widetilde{g}(x, y).
		\end{equation*}
	\end{prop}
	\begin{proof}
		If $x = 0$ or $y = 0$, we know that $H(\lfloor nx \rfloor, \lfloor ny \rfloor) = 0$ for all $n$, thus $\widetilde{g}(x, y) = 0$. Now we can assume that $x, y \in \mathbb{Q}_{> 0}$. Recall the definition of $X_{m, n}$ from \eqref{eq:Xmn}. By Propositions \ref{prop:item24} and \ref{prop:item1}, $\{X_{m, n}, m \leq n \in \mathbb{Z}_{\geq 0}\}$ satisfy the conditions of the superadditive ergodic theorem. Therefore, there exists a constant $\gamma(x, y)$ such that almost surely,
		$\lim_{n \to \infty} \frac{X_{0, n}}{n} = \gamma(x, y)$. By Proposition \ref{prop:equalind}, we know that almost surely, $\lim_{n \to \infty} \frac{H(nNx, nNy)}{n} = \gamma(x, y)$. Recall that $N$ is the positive integer defined in \eqref{eq:defN}. Since the height function $h: \mathbb{Z}_{\geq 0} \times \mathbb{Z}_{\geq 0} \to \mathbb{Z}_{\geq 0}$ is a Lipschitz function, we conclude that almost surely,
		\begin{equation*}
		\lim_{n \to \infty} \frac{H(\lfloor nx \rfloor, \lfloor ny \rfloor)}{n} = \gamma(x, y)/N := \widetilde{g}(x, y). \qedhere
		\end{equation*}
	\end{proof}
	
	\begin{proof}[Proof of Theorem \ref{thm:maincs6v}]
Using the fact that $\mathbb{Q}_{\geq 0}^2$ is countable and Proposition \ref{prop:onedirection}, almost surely,
		\begin{equation*}
		\lim_{n \to \infty} \frac{H(\lfloor nx \rfloor, \lfloor ny \rfloor)}{n} = \widetilde{g}(x, y)     
		\end{equation*}
		for all $(x, y) \in \mathbb{Q}_{> 0}^2$. By definition, $H$ is a Lipschitz function, hence $\widetilde{g}: \mathbb{Q}_{\geq 0}^2 \to \mathbb{R}$ is also a Lipschitz function. We can thus extend $\widetilde{g}$ to be a Lipschitz function on $\mathbb{R}_{\geq 0}^2$. By the density of $\mathbb{Q}_{\geq 0}^2$ in $\mathbb{R}_{\geq 0}^2$ and the Lipschitz property, we have that almost surely,
	    \begin{equation*}
		\lim_{n \to \infty} \frac{H(\lfloor nx \rfloor, \lfloor ny\rfloor)}{n} = \widetilde{g}(x, y)
		\end{equation*}
		for all $(x, y) \in \mathbb{R}_{\geq 0}^2$. Taking $g(x, y) = y - \widetilde{g}(x, y)$, we conclude the proof.
	\end{proof}
	
	\appendix
		\section{The Boolean-type product for the $t$-PNG model}\label{sec:Boolean-type product}
	Taking $b_1 = t$ and $b_2 = 1$, the set $W$ in \eqref{eq:setW} reduces to $\{0, t, 1-t, 1\}$. The Boolean-type product reduces to the following: for $a, b \in \{0, t, 1-t, 1\}$,
	\begin{equation*}
	a * b := 
	\begin{cases}
	0 &\qquad \text{ if } \{a, b\} = \{t, 1-t\},\\
	a & \qquad \text{ if } a = b,\\
	ab & \qquad \text{ else.} 
	\end{cases}
	\end{equation*}
Recall from \cite[Definition 1.7]{drillick2022hydrodynamics} that %that we define 
the $\LL^n$-matrix of the $t$-PNG model is defined as 
	\begin{equation*}
	\LL^n(\bi, \bj; \bk, \bl) := \mmin_{r \in \{1, \dots, n\}} \Big(\LL^1 (\mfs_r (\bi), \mfs_r (\bj); \mfs_r (\bk), \mfs_r (\bl))\Big),
	\end{equation*}
	where $\mmin$ is a modified version of the minimum. For $x_1, \dots, x_n \in \{0, t, 1-t, 1\}$,  
	\begin{equation*} 
	\mmin\big(x_1, \dots, x_n\big) := 
	\begin{cases}
	0 &\qquad \text{if $x_i = t$ and $x_j = 1-t$ for some $i,j \in \{1,\dots,n\}$,}\\
	\min\big(x_1, \dots, x_n\big) &\qquad \text{else.}\\
	\end{cases} 
	\end{equation*}
	It is straightforward to verify that $\mmin(a, b) = a * b$ for $a, b \in \{0, t, 1-t, 1\}$, and thus 
	$
	\mmin(x_1, \dots, x_n) = x_1 * \cdots * x_n.
	$
	Hence, the definition \cite[Definition 1.7]{drillick2022hydrodynamics} is equivalent to Definition \ref{def:nLmatrix} when $b_1 = t$ and $b_2 = 1$.
	\section{Two-colored CS6V configurations}
	\label{sec:twocolored}
	
	In this section, we list the configurations of the two-colored CS6V model. Note that the weights of the vertex configurations are given by the matrix $\LL^2$.
	\begin{figure}[ht]
		\centering
		\begin{tikzpicture}
		\begin{scope}[yshift = 2cm]
		\draw[thick, red] (-0.5, 0) -- (0.5, 0); 
		\node at (0, -0.8) {$1$};
		\end{scope}
		\begin{scope}[yshift = 2cm, xshift = 2cm]
		\draw[thick,red] (0, -0.5) -- (0, 0.5); 
		\node at (0, -0.8) {$1$};
		\end{scope}
		\begin{scope}[yshift = 2cm, xshift = 4cm]
		\draw[thick, red] (0, -0.5) -- (0, 0.5); 
		\draw[thick,red] (-0.5, 0) -- (0.5, 0);
		\node at (0, -0.8) {$b_1$};
		\end{scope}
		
		\begin{scope}[yshift = 2cm, xshift = 6cm]
		\draw[white] (0, -0.5) -- (0, 0.5); 
		\draw[white] (-0.5, 0) -- (0.5, 0);
		\node at (0, -0.8) {$b_2$};
		\end{scope}
		
		\begin{scope}[yshift = 2cm, xshift = 8cm]
		\draw[thick,red] (0, -0.5) -- (0, 0); 
		\draw[thick,red] (-0.5, 0) -- (0, 0);
		\node at (0, -0.8) {$1 - b_1$};
		\end{scope}
		
		\begin{scope}[yshift = 2cm, xshift = 10cm]
		\draw[thick,red] (0, 0) -- (0, 0.5); 
		\draw[thick,red] (0, 0) -- (0.5, 0);
		\node at (0, -0.8) {$1 - b_2$};
		\end{scope}
		
		\begin{scope}[yshift = 2cm, xshift = 12cm]
		\draw[thick,red] (-0.5, 0) -- (0.5, 0);
		\draw[thick,blue] (0, -0.5) -- (0, -0.1);
		\draw[thick,blue] (0, -0.1) -- (0.5, -0.1); 
		\node at (0, -0.8) {$1 - b_1$};
		\end{scope}
		\begin{scope}[yshift = 2cm, xshift = 14cm]
		\draw[thick,red] (-0.5, 0) -- (0.5, 0);
		\draw[thick,blue] (0, -0.5) -- (0, -0.1);
		\draw[thick,blue] (0, -0.5) -- (0, 0.5);
		\node at (0, -0.8) {$b_1$};
		\end{scope}
		\begin{scope}[xshift = 0cm]
		\draw[thick,blue] (-0.5, 0) -- (0.5, 0);
		\draw[thick,red] (-0.5, -0.1) -- (0.5, -0.1);
		\draw[thick,blue] (0, -0.5) -- (0, 0.5);
		\node at (0, -0.8) {$1$};
		\end{scope}
		\begin{scope}[xshift = 2cm]
		\draw[thick,blue] (-0.5, 0) -- (0.5, 0);
		\draw[thick,red] (-0.5, -0.1) -- (0.5, -0.1);
		\draw[thick,red] (0, -0.5) -- (0, 0.5);
		\node at (0, -0.8) {$b_1$};
		\end{scope}
		\begin{scope}[xshift = 4cm]
		
		\draw[thick,red] (0, -0.5) -- (0, -0.1);
		\draw[thick,blue] (-0.5, 0) -- (0, 0);
		\draw[thick,red] (-0.5, -0.1) -- (0, -0.1);
		%\draw[red] (0, 0) -- (0.5, 0);
		\draw[thick,blue] (0, 0) -- (0, 0.5);
		\node at (0, -0.8) {$1 - b_1$};
		\end{scope}
		\begin{scope}[xshift = 6cm]
		\draw[thick,blue] (-0.5, 0) -- (0.5, 0);
		\draw[thick,red] (0, -0.5) -- (0, 0.5);
		%	\draw[red] (-0.5, -0.1) -- (0.5, -0.1);
		%\draw[blue] (-0.5, -0.1) -- (0, -0.1);
		%\draw[blue] (0, -0.1) -- (0, 0.5);
		%\draw[red] (0, -0.5) -- (0, 0.5);
		\node at (0, -0.8) {$b_1$};
		\end{scope}
		\begin{scope}[xshift = 8cm]
		\draw[thick,blue] (-0.5, 0) -- (-0.1, 0);
		\draw[thick,red] (0, -0.5) -- (0, 0.5);
		%\draw[blue] (-0.5, -0.1) -- (0, -0.1);
		\draw[thick,blue] (-0.1, 0) -- (-0.1, 0.5);
		%\draw[red] (0, -0.5) -- (0, 0.5);
		\node at (0, -0.8) {$1 - b_1$};
		%	\draw[red] (-0.5, -0.1) -- (0.5, -0.1);
		\end{scope}
		\begin{scope}[xshift = 10cm]
		\draw[thick,blue] (-0.5, 0) -- (0.5, 0);
		\draw[thick,red] (-0.1, -0.5) -- (-0.1, 0.5);
		\draw[blue] (0, -0.5) -- (0, 0.5);
		%\draw[blue] (-0.5, -0.1) -- (0, -0.1);
		%\draw[red] (0, -0.5) -- (0, 0.5);
		%	\draw[red] (-0.5, -0.1) -- (0.5, -0.1);
		\node at (0, -0.8) {$1$};
		\end{scope}
		\begin{scope}[xshift = 12cm]
		\draw[thick,red] (-0.5, 0) -- (0.5, 0);
		\draw[thick,red] (-0.1, -0.5) -- (-0.1, 0.5);
		\draw[thick,blue] (0, -0.5) -- (0, 0.5);
		%\draw[blue] (-0.5, -0.1) -- (0, -0.1);
		%\draw[red] (0, -0.5) -- (0, 0.5);
		\node at (0, -0.8) {$b_1$};
		\end{scope}
		\begin{scope}[xshift = 14cm]
		%	\draw[red] (-0.5, 0) -- (-0.1, 0);
		\draw[thick,red] (-0.1, -0.5) -- (-0.1, 0) -- (-0.5, 0);
		\draw[thick,blue] (0, -0.5) -- (0, 0);
		\draw[thick,blue] (0, 0) -- (0.5, 0);
		%\draw[blue] (-0.5, -0.1) -- (0, -0.1);
		%\draw[red] (0, -0.5) -- (0, 0.5);
		\node at (0, -0.8) {$1 - b_1$};
		\end{scope}
		
		\begin{scope}[yshift = -2cm, xshift = 0cm]
		\draw[thick, red] (-0.5, 0) -- (0.5, 0);
		\draw[thick, blue] (-0.5, -0.1) -- (0.5, -0.1);
		%\draw[thick, red] (-0.1, -0.5) -- (-0.1, 0.5);
		%	\draw[thick, blue] (0, -0.5) -- (0, 0);
		%	\draw[thick, blue] (-0.5, 0) -- (0, 0);
		%\draw[thick, blue] (-0.5, -0.1) -- (0, -0.1);
		%\draw[thick, red] (0, -0.5) -- (0, 0.5);
		\node at (0, -0.8) {$b_2$};
		\end{scope}
		\begin{scope}[yshift = -2cm, xshift = 2cm]
		\draw[thick,blue] (-0.5, 0) -- (0, 0) -- (0, 0.5);
		\draw[thick,red] (-0.5, -0.1) -- (0.5, -0.1);
		\node at (0, -0.8) {$1- b_2$};
		\end{scope}
		\begin{scope}[yshift = -2cm, xshift = 4cm]
		%\draw[thick, red] (-0.1, -0.5) -- (-0.1, 0.5);
		%\draw[thick, blue] (0, -0.5) -- (0, 0.5);
		%	\draw[thick, red] (-0.5, 0) -- (0.5, 0);
		\draw[thick, blue] (-0.5, -0.1) -- (0.5, -0.1);
		%\draw[thick, red] (-0.1, -0.5) -- (-0.1, 0.5);
		%	\draw[thick, blue] (0, -0.5) -- (0, 0);
		%	\draw[thick, blue] (-0.5, 0) -- (0, 0);
		%\draw[thick, blue] (-0.5, -0.1) -- (0, -0.1);
		%\draw[thick, red] (0, -0.5) -- (0, 0.5);
		\node at (0, -0.8) {$b_2$};
		\end{scope}
		\begin{scope}[yshift = -2cm, xshift = 6cm]
		%\draw[thick, red] (-0.1, -0.5) -- (-0.1, 0.5);
		%\draw[thick, blue] (0, -0.5) -- (0, 0.5);
		\draw[thick, red] (0.5, -0.1) -- (0.1, -0.1) -- (0.1, 0.5);
		\draw[thick, blue] (-0.5, -0.1) -- (0, -0.1) -- (0, 0.5);
		%\draw[thick, red] (-0.1, -0.5) -- (-0.1, 0.5);
		%	\draw[thick, blue] (0, -0.5) -- (0, 0);
		%	\draw[thick, blue] (-0.5, 0) -- (0, 0);
		%\draw[thick, blue] (-0.5, -0.1) -- (0, -0.1);
		%\draw[thick, red] (0, -0.5) -- (0, 0.5);
		\node at (0, -0.8) {$1- b_2$};
		\end{scope}
		\begin{scope}[yshift = -2cm, xshift = 8cm]
		\draw[thick, red] (-0.1, -0.5) -- (-0.1, 0.5);
		\draw[thick, blue] (0, -0.5) -- (0, 0.5);
		%	\draw[thick, red] (-0.5, 0) -- (0.5, 0);
		%\draw[thick, red] (-0.1, -0.5) -- (-0.1, 0.5);
		%	\draw[thick, blue] (0, -0.5) -- (0, 0);
		%	\draw[thick, blue] (-0.5, 0) -- (0, 0);
		%\draw[thick, blue] (-0.5, -0.1) -- (0, -0.1);
		%\draw[thick, red] (0, -0.5) -- (0, 0.5);
		\node at (0, -0.8) {$b_2$};
		\end{scope}
		\begin{scope}[yshift = -2cm, xshift = 10cm]
		\draw[thick, red] (-0.1, -0.5) -- (-0.1, 0.5);
		\draw[thick, blue] (0, -0.5) -- (0, -0.1) -- (0.5, -0.1);
		%	\draw[thick, red] (-0.5, 0) -- (0.5, 0);
		%\draw[thick, red] (-0.1, -0.5) -- (-0.1, 0.5);
		%	\draw[thick, blue] (0, -0.5) -- (0, 0);
		%	\draw[thick, blue] (-0.5, 0) -- (0, 0);
		%\draw[thick, blue] (-0.5, -0.1) -- (0, -0.1);
		%\draw[thick, red] (0, -0.5) -- (0, 0.5);
		\node at (0, -0.8) {$1-b_2$};
		\end{scope}
		\begin{scope}[yshift = -2cm, xshift = 12cm]
		\draw[thick, blue] (-0.1, -0.5) -- (-0.1, 0.5);
		%\draw[thick, blue] (0, -0.5) -- (0, -0.1) -- (0.5, -0.1);
		%	\draw[thick, red] (-0.5, 0) -- (0.5, 0);
		%\draw[thick, red] (-0.1, -0.5) -- (-0.1, 0.5);
		%	\draw[thick, blue] (0, -0.5) -- (0, 0);
		%	\draw[thick, blue] (-0.5, 0) -- (0, 0);
		%\draw[thick, blue] (-0.5, -0.1) -- (0, -0.1);
		%\draw[thick, red] (0, -0.5) -- (0, 0.5);
		\node at (0, -0.8) {$b_2$};
		\end{scope}
		\begin{scope}[yshift = -2cm, xshift = 14cm]
		\draw[thick, blue] (-0.1, -0.5) -- (-0.1, 0) -- (0.5, 0);
		\draw[thick, red] (-0.1, 0.5) -- (-0.1, 0.1) -- (0.5, 0.1);
		%\draw[thick, blue] (0, -0.5) -- (0, -0.1) -- (0.5, -0.1);
		%	\draw[thick, red] (-0.5, 0) -- (0.5, 0);
		%\draw[thick, red] (-0.1, -0.5) -- (-0.1, 0.5);
		%	\draw[thick, blue] (0, -0.5) -- (0, 0);
		%	\draw[thick, blue] (-0.5, 0) -- (0, 0);
		%\draw[thick, blue] (-0.5, -0.1) -- (0, -0.1);
		%\draw[thick, red] (0, -0.5) -- (0, 0.5);
		\node at (0, -0.8) {$1 - b_2$};
		\end{scope}
		\begin{scope}[yshift = -4cm, xshift = 0cm]
		\draw[thick, blue] (-0.1, -0.5) -- (-0.1, 0.5);
		\draw[thick, blue] (-0.5, -0.1) -- (0, -0.1) -- (0.5, -0.1);
		%	\draw[thick, red] (-0.5, 0) -- (0.5, 0);
		%\draw[thick, red] (-0.1, -0.5) -- (-0.1, 0.5);
		%	\draw[thick, blue] (0, -0.5) -- (0, 0);
		%	\draw[thick, blue] (-0.5, 0) -- (0, 0);
		%\draw[thick, blue] (-0.5, -0.1) -- (0, -0.1);
		%\draw[thick, red] (0, -0.5) -- (0, 0.5);
		\node at (0, -0.8) {$b_1 b_2$};
		\end{scope}
		\begin{scope}[yshift = -4cm, xshift = 2cm]
		\draw[thick, blue] (-0, -0.5) -- (-0, 0.5);
		\draw[thick, blue] (-0.5, -0.1) -- (0, -0.1) -- (0.5, -0.1);
		\draw[thick, red] (0.5, 0) -- (0.1, 0) -- (0.1, 0.5);
		%\draw[thick, red] (-0.1, -0.5) -- (-0.1, 0.5);
		%	\draw[thick, blue] (0, -0.5) -- (0, 0);
		%	\draw[thick, blue] (-0.5, 0) -- (0, 0);
		%\draw[thick, blue] (-0.5, -0.1) -- (0, -0.1);
		%\draw[thick, red] (0, -0.5) -- (0, 0.5);
		\node at (0, -0.8) {$(1-b_1)(1 - b_2)$};
		\end{scope}
		\begin{scope}[yshift = -4cm, xshift = 4cm]
		\draw[thick, blue] (0, -0.5) -- (0, 0) -- (-0.5, 0);
		\draw[thick, red] (0, 0.5) -- (0, 0) -- (0.5, 0);
		%\draw[thick, red] (-0.1, -0.5) -- (-0.1, 0.5);
		%	\draw[thick, blue] (0, -0.5) -- (0, 0);
		%	\draw[thick, blue] (-0.5, 0) -- (0, 0);
		%\draw[thick, blue] (-0.5, -0.1) -- (0, -0.1);
		%\draw[thick, red] (0, -0.5) -- (0, 0.5);
		\node at (0, -0.8) {$b_1(1 - b_2)$};
		\end{scope}
		\begin{scope}[yshift = -4cm, xshift = 6cm]
		\draw[thick, blue] (-0, -0.5) -- (-0, 0);
		\draw[thick, blue] (-0.5, 0) -- (0, 0);
		%	\draw[thick, red] (-0.5, 0) -- (0.5, 0);
		%\draw[thick, red] (-0.1, -0.5) -- (-0.1, 0.5);
		%	\draw[thick, blue] (0, -0.5) -- (0, 0);
		%	\draw[thick, blue] (-0.5, 0) -- (0, 0);
		%\draw[thick, blue] (-0.5, -0.1) -- (0, -0.1);
		%\draw[thick, red] (0, -0.5) -- (0, 0.5);
		\node at (0, -0.8) {$(1 - b_1)b_2$};
		\end{scope}
		\begin{scope}[yshift = -4cm, xshift = 8cm]
		\draw[thick,red] (-0.5, -0.1) -- (0.5, -0.1);
		\draw[thick,red] (-0.1, -0.5) -- (-0.1, 0.5);
		\draw[thick,blue] (0, -0.5) -- (0, 0.5);
		\draw[thick,blue] (-0.5, 0) -- (0.5, 0);
		%\draw[blue] (-0.5, -0.1) -- (0, -0.1);
		%\draw[red] (0, -0.5) -- (0, 0.5);
		\node at (0, -0.8) {$b_1 b_2$};
		\end{scope}
		\begin{scope}[yshift = -4cm, xshift = 10cm]
		\draw[thick,red] (-0.1, -0.5) -- (-0.1, -0.1);
		%	\draw[red] (-0.5, -0.1) -- (-0.1, -0.1);
		\draw[thick,red] (0, -0.5) -- (0, 0.1) -- (-0.5, 0.1);
		%\draw[red] (-0.1, -0.5) -- (-0.1, 0.5);
		\draw[thick,blue] (-0.1, -0.5) -- (-0.1, 0);
		\draw[thick,blue] (-0.5, 0) -- (-0.1, 0);
		%\draw[blue] (-0.5, -0.1) -- (0, -0.1);
		%\draw[red] (0, -0.5) -- (0, 0.5);
		\node at (0, -0.8) {$(1-b_1)b_2$};
		\end{scope}
		\begin{scope}[yshift = -4cm, xshift = 12cm]
		\draw[thick,blue] (-0.1, -0.5) -- (-0.1, 0);
		\draw[thick,red] (-0.5, 0.1) -- (0.5, 0.1);
		\draw[thick,red] (0, -0.5) -- (0, 0.5);
		\draw[thick,blue] (-0.5, 0) -- (-0.1, 0);
		%\draw[blue] (-0.5, -0.1) -- (0, -0.1);
		%\draw[red] (0, -0.5) -- (0, 0.5);
		\node at (0, -0.8) {$(1-b_2) b_1$};
		\end{scope}
		\begin{scope}[yshift = -4cm, xshift = 14cm]
		\draw[thick,blue] (0, -0.5) -- (0, 0.5);
		%	\draw[red] (-0.5, -0.1) -- (0.5, -0.1);
		\draw[thick,red] (-0.1, -0.5) -- (-0.1, -.1) -- (-0.5, -.1);
		\draw[thick,blue] (-0.5, 0) -- (0.5, 0);
		%\draw[blue] (-0.5, -0.1) -- (0, -0.1);
		%\draw[red] (0, -0.5) -- (0, 0.5);
		\node at (0, -0.8) {$(1-b_1)(1-b_2)$};
		\end{scope}
		\end{tikzpicture}
		\caption{The two-colored CS6V configurations with non-zero weight given by $\mathsf{L}^2$. Red has a higher priority than blue. }
		\label{fig:twocolor}
	\end{figure}
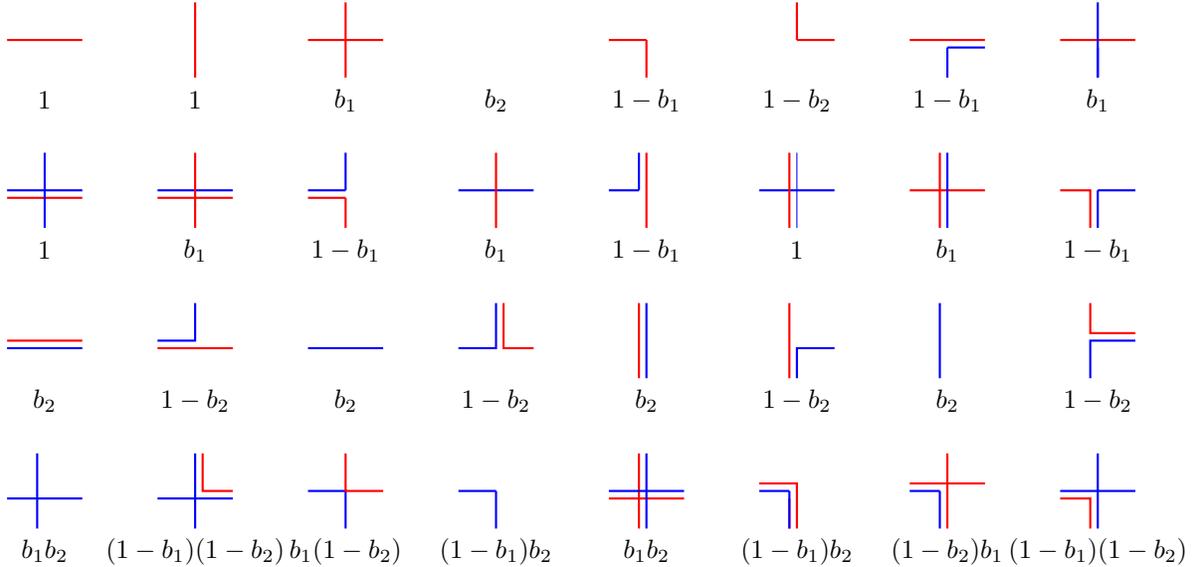

	\bibliographystyle{alpha}
	\bibliography{refs.bib}
	%\nocite{*}
\end{document}